\documentclass[11pt]{amsart}
\usepackage{amsthm}
\usepackage{enumitem}
\usepackage{moreenum}
\usepackage{amssymb,verbatim}
\usepackage{upgreek}
\usepackage{xspace,cmap}
\usepackage{csquotes}
\usepackage{stmaryrd} 
\usepackage[colorlinks,citecolor=blue,urlcolor=black,linkcolor=black]{hyperref}
\usepackage{pst-node}
\usepackage{tikz-cd} 

\newtheorem*{rep@theorem}{\rep@title}
\newcommand{\newreptheorem}[2]{%
\newenvironment{rep#1}[1]{%
 \def\rep@title{#2 \ref{##1}}%
 \begin{rep@theorem}}%
 {\end{rep@theorem}}}

\renewcommand{\restriction}{\mathbin\upharpoonright}    

\newtheorem*{theorem*}{Theorem}
\newtheorem*{maintheorem*}{Main Theorem}
\newtheorem*{question*}{Question}
\newtheorem*{corollary*}{Corollary}
\newtheorem*{definition*}{Definition}

\newtheorem{theorem}{Theorem}[section]
\newreptheorem{theorem}{Theorem}

\newtheorem{prop}[theorem]{Proposition}
\newtheorem{claim}{Claim}[theorem]
\newtheorem{subclaim}{Subclaim}[claim]
\newtheorem{lemma}[theorem]{Lemma}
\newtheorem{cor}[theorem]{Corollary}

\newtheorem{question}{Question}
\newtheorem{fact}[theorem]{Fact}

\theoremstyle{definition}

\newtheorem{definition}[theorem]{Definition}

\newtheorem*{setup}{Setup~\thesection}

\theoremstyle{remark}
\newtheorem{remark}[theorem]{Remark}

\newcommand\HOD{\textnormal{HOD}}

\newcommand\diagonal{\bigtriangleup}

\newcommand\ale[1]{\marginpar{Alejandro: #1}}

\DeclareMathOperator{\supp}{supp}
\DeclareMathOperator{\crit}{crit}

\DeclareMathOperator{\Ult}{Ult}
\DeclareMathOperator{\id}{id}

    \def\sq{\sqsubseteq}
    
    \newcommand{\one}{\mathop{1\hskip-3pt {\rm l}}}

\makeatletter
\newcommand{\tpitchfork}{%
  \vbox{
    \baselineskip\z@skip
    \lineskip-.52ex
    \lineskiplimit\maxdimen
    \m@th
    \ialign{##\crcr\hidewidth\smash{$-$}\hidewidth\crcr$\pitchfork$\crcr}
  }%
}
\makeatother

\def\s{\subseteq}
\def\forces{\Vdash}

\DeclareMathOperator{\otp}{otp}

\DeclareMathOperator{\cf}{cf}

\DeclareMathOperator{\ord}{Ord}
\DeclareMathOperator{\add}{Add}

\renewcommand{\mid}{\mathrel{|}\allowbreak}

\newcommand{\dom}{\mathop{\mathrm{dom}}\nolimits}

\title[On the optimality of the HOD dichotomy]{On the optimality of the HOD dichotomy}
\author[Goldberg]{Gabriel Goldberg}
\author[Osinski]{Jonathan Osinski}
\author[Poveda]{Alejandro Poveda}

\thanks{The results presented in this paper are---with the exception of Theorem~\ref{ccenonextendible}---due to Goldberg and Poveda. Theorem~\ref{ccenonextendible} appears in Osinski's Ph.D. thesis \cite{osinski2024} and was proved  with Poveda during a visit to Harvard in 2024. Goldberg was partially supported by NSF Foundations Grant 2401789. Poveda acknowledges support from Harvard University through the Department of Mathematics and the Center of Mathematical Sciences and Applications. }

\address[Goldberg]{Department of Mathematics, UC Berkeley, CA 94720, USA}
\email{ggoldberg@berkeley.edu}
\address[Osinski]{Hamburg University, Department of Mathematics,  Bundesstraße 55,
20146 Hamburg, Germany}
\email{jonathan.osinski@uni-hamburg.de}
\address[Poveda]{ Department of Mathematics and Center of Mathematical Sciences and Applications, Harvard University, MA, 02138, USA. \textbf{Current address:} Departamento de Matemáticas, CUNEF Universidad, Campus  Almansa,  Madrid, 28040, Spain}

\email{alejandro.poveda@cunef.edu}
\begin{document}
\maketitle

\begin{abstract}
 In the first part of the manuscript, we establish several consistency results concerning Woodin's $\HOD$ hypothesis and large cardinals around the level of extendibility. First, we prove that the first extendible cardinal can be the first strongly compact in HOD. Second, answering a question in \cite{CFH}, we prove that the first supercompact cardinal can be the first strongly compact, non supercompact, of $\HOD$. We extend a former result of Woodin by showing that under the HOD hypothesis the first extendible cardinal is $C^{(1)}$-supercompact in HOD. We also show that the first cardinal-correct extendible may not be extendible, thus answering a question by Gitman and Osinski \cite[\S9]{GitOsi}.

   In the second part of the manuscript, we discuss the extent to which weak covering can fail below the first supercompact cardinal $\delta$ in a context where the HOD hypothesis holds. Answering a question of Cummings et al. \cite{CumFriGol}, we show that under the $\HOD$ hypothesis there are many singulars $\kappa<\delta$ where $\cf^{\HOD}(\kappa)=\cf(\kappa)$ and $\kappa^{+\HOD}=\kappa^{+}.$ In contrast, we also show that the $\HOD$ hypothesis is consistent with $\delta$ carrying a club of $\HOD$-regulars cardinals $\kappa$ such that  $\kappa^{+\HOD}<\kappa^{+}$.  Finally, we close the manuscript with a discussion about the $\HOD$ hypothesis and $\omega$-strong measurability.
\end{abstract}


\section{Introduction}
When is a mathematical object definable? Can this intuitive concept be formalized at all?
On first thought, the answer is no. 
Suppose the concept of definability could itself be precisely defined.
Consider the least undefinable ordinal number. Since there are uncountably many ordinals and only countably many definitions, this ordinal must exist.
However the number is definable, contrary to its definition!

\smallskip

In a 1946 Princeton lecture \cite{Godel}, G\"odel set out to circumvent this paradox, arriving at the concept of ordinal definability. Intuitively, a set $X$ is \emph{ordinal definable} if it is the unique object satisfying some set-theoretic property that can be defined in terms of finitely many ordinal numbers. A set $X$ is \emph{hereditarily ordinal definable} (in symbols, $X\in \HOD$) if $X$ is ordinal definable, every element of $X$ is ordinal definable, every element of every element of $X$ is ordinal definable, and so on; see \S\ref{sec: prelimminaries} for a  formal definition. After Gödel, subsequent investigation by 
Myhill--Scott \cite{Myhill}, Vop\v{e}nka \cite{Vopenka}, and others began to outline the theory of HOD. 
In last few years, the study of HOD has experienced renewed interest after groundbreaking findings due to  Woodin \cite{WooPartI,WoodinPartII,WoodinDavisRodriguez, midrasha}. 

\smallskip

$\HOD$ is a transitive class containing all the ordinals and satisfying ZFC, the Zermelo-Fraenkel axioms of set theory together with the Axiom of Choice. In modern set-theoretic terminology, $\HOD$ is an \emph{inner model}. \emph{Inner model theory} is the major area of research in set theory launched by Gödel's discovery of the \emph{constructible universe}, $L$ \cite{GodelL}. Ever since set theorist like Jensen, Mitchell, Neeman, Steel and Woodin (see \cite{MitchHandbook}) have produced a hierarchy of  fine-structural inner models extending Gödel's  $L$. 
This hierarchy is organized according to the extent to which an inner model captures the large cardinal hierar\-chy of the  mathematical universe $V$. Although 
$\HOD$ is not regarded as a canonical inner model --- for example, it is malleable under the \emph{method of forcing} --- it contains all canonical inner models constructed to date and is sufficiently rich to accommodate all known large cardinals.

\emph{Large cardinals axioms} postulate the existence of infinite cardinals whose existence cannot be proved by ZFC. These axioms form a hierarchy of principles that permit for classification of virtually all mathematical theories extending ZFC according to their \emph{consistency strength}. Large cardinals have also helped to tackle deep problems concerning \emph{down-to-earth} mathematical objects; such as whether every subset of the real line can be Lebesgue measurable \cite{Sol},  or whether every almost free  abelian group must be necessarily free \cite{MagShegroups}. These axioms play a pivotal role in the foundations of mathematics by testing the limits of mathematical theories and providing deep insights into the independence phenomenon  \cite{Kan, Koel,Maddy}.


\smallskip

Woodin's work on HOD \cite{WooPartI, midrasha} concerns the question of whether $\HOD$ is close to resembling the true mathematical universe. In the presence of strong enough large cardinals, his \textit{HOD hypothesis} implies that the answer is yes: the two models closely resemble one another in terms of their large cardinal structure. This paper examines the limits on this resemblance, the extent to which HOD and \(V\) can differ even under the assumption of the HOD hypothesis.
Before turning to this, let us define the HOD hypothesis.

\begin{definition*}[\cite{midrasha}]
    The \textit{$\HOD$ hypothesis} states that there is a proper class of regular cardinals that are not \(\omega\)-strongly measurable in \(\HOD\).\footnote{For relevant definitions, see \S\ref{sec: prelimminaries}.}
\end{definition*}
The above hypothesis was formulated in the wake of the discovery of groundbreaking $\HOD$ dichotomy: 
\begin{theorem*}[HOD Dichotomy, \cite{midrasha}]\label{theo:HOD Dichotomy}
    If \(\delta\) is an extendible cardinal then exactly one of the following holds:
    \begin{enumerate}
      \item \(\HOD\) is a weak extender model for the supercompactness of \(\delta\).
        \item Every regular cardinal \(\kappa \geq \delta\) is \(\omega\)-strongly measurable in \(\HOD\).
    \end{enumerate}
\end{theorem*}

The $\HOD$ dichotomy has placed the foundations of mathematics at a critical crossroad. In simple terms, the dichotomy states that, assuming sufficiently strong large cardinals exist (i.e., extendibles),  either there are reasonable prospects for completing Gödel's inner model program, or there is no fine-structural insight into the mathematical  universe $V$  \cite{BKW}. A natural question thus emerges: which of these scenarios  prevails? The $\HOD$ hypothesis predicts that the prevailing scenario is the first one.

\smallskip

A corollary of the HOD dichotomy is that, under the HOD hypothesis,  every extendible cardinal is  supercompact in \(\HOD\).
This raises an interesting question: could one perhaps prove that under the \(\HOD\) hypothesis every extendible cardinal is extendible in \(\HOD\)? In fact, it is a consequence of  \emph{Woodin's Universality Theorem} (see Theorem \ref{UniversalityI} below)
that every extendible cardinal except perhaps the first one is extendible in \(\HOD\). Woodin's theorem is in stark contrast to classical fine-structural inner models, where the extent of similarity between these models and the mathematical universe is contingent upon so-called \emph{anti-large cardinal hypotheses}.

\smallskip

Our first theorem shows that 
Woodin's theorem is optimal:
\begin{reptheorem}{thm:HOD_thm}
     It is consistent with the \(\HOD\) hypothesis that the least extendible cardinal is the least 
    strongly compact cardinal in \(\HOD\). 
\end{reptheorem}
Thus the  first extendible is genuinely differ\-ent from the rest.\footnote{In Goldberg's paper \cite{Gol}, it is claimed that Woodin proved that under the HOD hypothesis, every extendible cardinal must be extendible in \(\HOD\). This is the negation of the theorem  stated above. In fact, Woodin never claimed to prove such a thing, and this was Goldberg's mistake alone.}

\smallskip

Continuing this vein, our second theorem answers a question of Cheng–Fri\-edman–Hamkins \cite{CFH} by showing that the first supercompact cardinal may not be supercompact in $\HOD$.
\begin{reptheorem}{theorem:main}
It is consistent  that the least supercompact cardinal is, in \(\HOD\), the least strongly compact cardinal but it is not supercompact.
\end{reptheorem}
The third theorem is that under the $\HOD$ hypothesis, even though the first extendible may be the first supercompact cardinal of \(\HOD\), there is still a sense in which it must be a rather large supercompact cardinal:
\begin{reptheorem}{HODabsorbsC(1)supercompacts}
   Under the \(\HOD\) hypothesis, if a cardinal $\delta$ is extendible then it must be $C^{(1)}$-supercompact in \(\HOD\).
\end{reptheorem}
The class of $C^{(n)}$-supercompacts was introduced by Bagaria \cite{Bag} and it has been extensively studied by Poveda in \cite{HMP,PovOmega, PovAxiomA}. Theorem~\ref{HODabsorbsC(1)supercompacts} above improves a former result by Woodin saying that under the HOD conjecture, the first extendible cardinal is supercompact in HOD. Incidentally, Theorems~\ref{thm:HOD_thm} and \ref{HODabsorbsC(1)supercompacts} 
provide an alternative proof of  the main identity crisis theorem of Hayut, Magidor and Poveda \cite{HMP}.


\smallskip

Following up on this issue, in  \S\ref{SectionTransferring}  we show that under the HOD hypothesis the classes of $C^{(n)}$-supercompact/$C^{(n)}$-extendible cardinals  above the first extendible transfer down to $\HOD$. More precisely, we show the following:
\begin{reptheorem}{TrasnferringCnextendibles}
    Suppose  $\delta$ is the first extendible and that the $\HOD$ hypothesis holds. If  $\kappa>\delta$ is $C^{(n)}$-supercompact (resp. $C^{(n)}$-extendible) then $\kappa$ is $C^{(n)}$-supercompact (resp. $C^{(n)}$-extendible) in $\HOD.$
\end{reptheorem}

We conclude this first part of the manuscript discussing a related problem. Inspired by previous work of Magidor and V\"a\"an\"aanen \cite{MagVan} and Galeotti, Khomskii and V\"a\"an\"aanen \cite{gal2020}, Gitman and Osinski \cite{GitOsi} introduced cardinal-correct extendible cardinals: A cardinal  $\delta$ is \emph{cardinal-correct extendible}  if for each ordinal $\alpha>\delta$ there is an elementary embedding $j\colon V_\alpha\rightarrow M$ such that $\crit(j)=\delta$, $j(\delta)>\alpha$ and $M$ is \emph{cardinal-correct} (i.e., $\text{Card}^M=\text{Card}\cap M$). It will be evident from the definition of extendibility (see page~\pageref{extendible}) that every extendible is cardinal-correct extendible. It is  asked in work by Gitman and Osinski \cite[\S9]{GitOsi} whether the opposite implication is true as well. 
We show that the answer is negative.
\begin{reptheorem}{ccenonextendible}
   It is consistent with the $\HOD$ hypothesis that the first cardinal-correct extendible is the first strongly compact cardinal.
\end{reptheorem}

This is a consequence of Theorem~\ref{thm:HOD_thm} and the following general result:

\begin{reptheorem}{thm:CCEandHOD}
    Suppose that $\HOD$ is cardinal-correct. Then every extendible cardinal is  
   cardinal-correct extendible in $\HOD$.  
\end{reptheorem}

\smallskip

In the second part of the paper we analyze the optimality of the HOD dichotomy for strong compact cardinals proved by the first author in \cite{Gol}. Namely, if $\delta$ is strong compact and the $\HOD$ hypothesis holds,  
every singular cardinal $\lambda>\delta$ is singular in $\HOD$ and $(\lambda^+)^{\HOD}=\lambda^+$ (i.e., \emph{weak covering holds} at $\lambda$).  A natural question is whether the assumption $``\delta$ is strong compact + $\HOD$ hypothesis''  is consistent with  many singular cardinals $\lambda<\delta$ violating weak covering.  
Building upon previous work of Cummings, Friedman and Golshani \cite{CumFriGol} we improve their main result showing that:
\begin{reptheorem}{thm: extendingCumFriGol}
  The following configuration is consistent with the $\HOD$ hypothesis: The first supercompact $\delta$ carries a club  $C\s \delta$ of consisting of $\HOD$-regular cardinals $\kappa$ such that $\kappa^{+\HOD}<\kappa^+$. 
\end{reptheorem}

Answering a question in \cite{CumFriGol}, we  show that this is arguably optimal:
\begin{reptheorem}{thm: many instances of covering}
    If $\delta$ is supercompact and the $\HOD$ hypothesis holds then there are unboundedly many singular cardinals $\kappa<\delta$ that are singular in $\HOD$ and $\kappa^{+\HOD}=\kappa^+$.
\end{reptheorem}

We conclude the paper with a brief discussion on the $\HOD$ dichotomy and $\omega$-strongly measurability in $\HOD$. 
In \cite{Gol}, Goldberg showed that if the $\HOD$ hypothesis fails and there is a strongly compact cardinal $\delta$, then all sufficiently large regular cardinals are $\omega$-strongly measurable in $\HOD$. Let $\eta_0$ be the least ordinal greater than or equal to $\delta$ such that all regular cardinals $\theta\geq \eta_0$ are $\omega$-strongly measurable in $\HOD$. Is it possible to bound $\eta_0$? We show that the answer is no. This shows that the refinement of Woodin's $\HOD$ dichotomy proved in \cite{Gol} cannot be significantly improved:
\begin{reptheorem}{thm: omegastrongly}
    Suppose \(\delta\) is supercompact
    and the \(\HOD\) hypothesis fails. Then for any ordinal \(\alpha\), there is a forcing extension \(V[G]\) in which \(\delta\) is supercompact, the \(\HOD\) hypothesis fails, and \(\eta_0^{V[G]} \geq \alpha\).
\end{reptheorem}
 The manuscript is fairly self-contained and the notation is standard in set theory. In \S\ref{sec: prelimminaries},  we provide  all the pertinent preliminaries. In \S\S\ref{SectionFirstExtendible}, \ref{SectionTransferring} and \ref{sec: CCE} we prove Theorems~\ref{thm:HOD_thm} through \ref{thm:CCEandHOD} above. 
 The first part of Section~\ref{sec: on the HOD dichotomy for supercompact} is devoted to proving Theorems~\ref{thm: extendingCumFriGol} and \ref{thm: many instances of covering}, while in the second part we will prove Theorem~\ref{thm: omegastrongly}. 
We also collect a few  open problems through the paper. 

\section{Preliminaries and notations}\label{sec: prelimminaries}
Recall that a set $X$ is \emph{ordinal definable} (in symbols,  $X\in \mathrm{OD}$) if there is a formula $\varphi(x,x_0,\dots, x_n)$ of the language of set theory  which, together with  ordinals $\alpha_0,\dots,\alpha_n$, defines $X$; namely, $X=\{x: \varphi(x,\alpha_0,\dots,\alpha_n)\}$. Additionally, $X$ is  \emph{Hereditarily Ordinal Definable} (in symbols, $X\in \HOD$) if $X\in \mathrm{OD}$ and the transitive closure of $\{X\}$\footnote{Recall that the transitive closure of $\{X\}$ is the smallest transitive set containing $\{X\}$.} consists of $\mathrm{OD}$-sets only.  A cardinal \(\kappa\) is called \textit{\(\omega\)-strongly measurable in \(\HOD\)}
if for some \(\eta\) with \((2^\eta)^\HOD < \kappa\), there is no partition in $\HOD$
of \(S^{\kappa}_\omega := \{\xi < \kappa: \cf(\xi) = \omega\}\) into \(\eta\)-many
stationary sets.  An inner model \(M\) of ZFC is a \textit{weak extender model}
for the supercompactness of a cardinal \(\kappa\)
if for all \(X\in M\), there is a normal fine \(\kappa\)-complete ultrafilter
\(\mathcal U\) on \(\mathcal{P}_\kappa(X)\cap M\) with \(\mathcal U\cap M\in M\). 
In particular, in this situation \(\kappa\) is a supercompact cardinal in both \(M\) and \(V\). Woodin has showed that if $\delta$ is extendible and the $\HOD$ hypothesis holds then $\HOD$ is a weak extender model for the supercompactness of $\delta$ \cite{WooPartI}.


\subsection{Some large cardinal notions}\label{PreliminariesLargeCardinals}

Here we collect a few relevant large-cardinal notions. For further information  we refer to Kanamori's  \cite{Kan}.

\begin{definition}[\cite{Bag}]
    For each $n<\omega$, denote by $C^{(n)}$ the class of all ordinals $\theta$ that are \emph{$\Sigma_n$-correct} (in symbols, $V_\theta\prec_{\Sigma_n} V$); namely, the ordinals for which given $\varphi(\bar{x})$ a $\Sigma_n$ formula in the language of set theory and $\bar{a}\in V_\theta$, 
    $$V_\theta\models \varphi(\bar{a})\;\Longleftrightarrow\; V\models \varphi(\bar{a}).$$
\end{definition}
 $C^{(0)}$ and  $C^{(1)}$ are, respectively, the collections of all ordinals and $\beth$-fixed points.  For each $n<\omega$, $C^{(n)}$ is a $\Pi_n$ definable club class. For additional information on these classes we refer the reader to \cite[\S2]{Bag}.

\smallskip

Two large-cardinal notions in which we will be interested are:

 \begin{definition}[Bagaria {\cite{Bag}}]\label{def: Cnsupercompactness}
    A cardinal $\delta$ is called \emph{$\lambda$-$C^{(n)}$-supercompact}  for  $\lambda>\delta$ if there is an elementary embedding $j\colon V\rightarrow M$ such that $\crit(j)=\delta$,  $j(\delta)>\lambda$, $M$ is a transitive class with $M^\lambda\s M$ and $j(\delta)\in C^{(n)}$. 
    
    $\delta$ is  \emph{$C^{(n)}$-supercompact} if it is $\lambda$-$C^{(n)}$-supercompact for all $\lambda>\delta$.
    \end{definition}

    \begin{definition}[Bagaria {\cite{Bag}}]
    A cardinal $\delta$ is called \emph{$\lambda$-$C^{(n)}$-extendible}  for  $\lambda>\delta$ if there is $\theta\in\mathrm{Ord}$ and an elementary embedding $j\colon V_\lambda\rightarrow V_\theta$ such that $\crit(j)=\delta$, $j(\delta)>\lambda$ and $j(\delta)\in C^{(n)}$. 
    
    $\delta$ is  \emph{$C^{(n)}$-extendible} if it is $\lambda$-$C^{(n)}$-extendible for all $\lambda>\delta$.
    \end{definition}

     A cardinal $\delta$ is supercompact  if it witnesses the property described in Definition~\ref{def: Cnsupercompactness} except for $j(\delta)\in C^{(1)}$ (cf. \cite{Kan}). Likewise, $\delta$ is extendible \label{extendible}if and only if it is $C^{(1)}$-extendible. $C^{(n)}$-extendibility yields a strong hierarchy in the sense that the first $C^{(n+1)}$-extendible is a limit of $C^{(n)}$-extendibles. In contrast, this does not apply to the class of $C^{(n)}$-supercompact cardinals (see \cite{HMP, PovAxiomA} or Corollary~\ref{cor: identity crises} below).

     \smallskip

     The next theorems characterize both  $C^{(n)}$-supercompactness and $C^{(n)}$-extend\-ibility in terms of ``tallness" \cite{HamkinsTall}:
\begin{theorem}[Poveda, {\cite[\S3]{PovAxiomA}}]\label{CharacterizingCnsupercompacts}
   The following are equivalent:
   \begin{enumerate}
       \item   $\delta$ is $C^{(n)}$-supercompact;
       \item   $\delta$ is supercompact and for each regular  $\lambda > \delta$ there is $i\colon V\rightarrow N$ such that $\crit(i)=\delta$, $N^{{<}\delta}\s N$, $\cf(i(\delta))>\lambda$, and $i(\delta)\in C^{(n)}$.
   \end{enumerate}
\end{theorem}

\begin{theorem}[Bagaria and Goldberg, {\cite{BagGol}}]\label{CharacterizingExtendibility}
    A cardinal $\delta$ is extendible if and only if for arbitrary large   $\lambda>\delta$, there is a normal fine  $\delta$-complete ultrafilter on $T_{\delta,\lambda}:=\{x\in \mathcal{P}_\delta(\lambda): \text{$\otp(x)$ is $\Sigma_2$-correct}\}$.
\end{theorem}

    It is not hard to show that if $\delta$ is extendible and $\lambda>\delta$ is an \emph{indestructibly $\Sigma_2$-correct cardinal} (i.e., $\one\forces_{\mathbb{P}}``\lambda\in C^{(2)}"$ for all $\lambda$-directed-closed poset $\mathbb{P}$) then there is a normal ultrafilter on the following set $$T^*_{\delta,\lambda}:=\{x\in \mathcal{P}_\delta(\lambda): \text{$\otp(x)$ is indestructibly $\Sigma_2$-correct}\}.$$

An interesting question in regards to the HOD dichotomy is whether  a large cardinal hypothesis weaker than extendibility can be employed to prove that \(\HOD\) is a weak extender
model for the supercompactness of some cardinal. A solution to this problem was provided by  Woodin:

\begin{definition}[Woodin {\cite[\S7.1]{WooPartI}}]
    A cardinal \(\delta\) is \emph{\(\HOD\)-supercompact}
if for all cardinals \(\lambda\), there is an elementary embedding \(j : V\to M\)
with \(\crit(j) = \delta\), 
\(M^\lambda \subseteq M\), 
\(j(\delta) > \lambda\), and 
\(\HOD^M\cap V_\lambda = \HOD\cap V_\lambda\).
\end{definition}

\begin{theorem}[{\cite[Theorem~193]{WooPartI}}]\label{thm:WoodinHODsupercompact}
    Under the \(\HOD\) hypothesis, if a cardinal \(\delta\) is \(\HOD\)-supercompact, then \(\HOD\) is a weak extender model
    for the supercompactness of \(\delta\) (in particular, $\delta$ is supercompact in $\HOD$).
\end{theorem}

Every extendible cardinal is \(\HOD\)-supercompact \cite[\S7]{WooPartI}, but the consistency strength of \(\HOD\)-supercompactness is
the same as supercompactness\footnote{One can force \(V = \HOD\) while preserving all supercompacts,
and if \(V = \HOD\), then every supercompact cardinal is trivially \(\HOD\)-supercompact.}. However how do these two notions relate implication-wise?  For instance, a priori it is not far-fetched that every extendible cardinal is a limit of $\HOD$-supercompacts. We will answer this question in the negative in the forthcoming Corollary~\ref{cor: identity crises}.

\subsection{The Universality Theorems}\label{sec: Universality Theorems}
For future convenience, in this section we  formulate the Universality Theorems for inner models $N$ containing the ordinals and having  the so-called \emph{covering} and \emph{approximation} properties. 

\begin{definition}[Hamkins, {\cite{HamCover}}]
    Suppose that $\delta>\omega$ is a regular cardinal and that $N$ is a transitive inner model of $\mathrm{ZFC}$ containing the ordinals. 
    \begin{enumerate}
        \item $N$ has the \emph{$\delta$-cover property} if for all $\sigma\s N$ with $|\sigma|<\delta$, there exists $\tau\in N$ such that $\sigma\s \tau$ and  $|\tau|<\delta$.
        \item $N$ has the \emph{$\delta$-approximation property} if for all $X\s N$, 
             $X\in N$ if and only if 
            $X\cap \tau\in N$ for all $\tau\in N$ with $|\tau|<\delta$.
    \end{enumerate}
\end{definition}

The next theorem due to Hamkins shows that if $N$ has the $\delta$-cover and $\delta$-approximation properties then $N$ is easily definable. Thus, the theory of such inner models is part of the first order theory of  $V$.

\begin{theorem}[Definability Theorem, Hamkins \cite{Geology}]\label{Definability}
  Suppose  $\delta$ is a regular uncountable cardinal and that $N$ is an inner model of $\mathrm{ZFC}$ containing all the ordinals and having the $\delta$-covering and $\delta$-approximation properties. Then $N$ is $\Sigma_2$ definable with $N\cap H(\delta^+)$ as a parameter. 
\end{theorem}

Hamkins and Woodin have showed that inner models $N$ with the approximation and cover properties inherit most of the large cardinal structure of $V$. A precise formulation of this fact  is provided by the \emph{Universality Theorems}:
\begin{theorem}[Universality Theorem I, Woodin]\label{UniversalityI}
  Suppose  $\delta$ is a regular uncountable cardinal and that $N$ is an inner model of $\mathrm{ZFC}$ containing all the ordinals and having the $\delta$-covering and $\delta$-approximation properties. Suppose that $E$ is an $N$-extender of length $\eta$ with critical point $\kappa_E\geq \delta$. Let
  $$j_{E}\colon N\rightarrow N_E$$
  be the ultrapower embedding. Then the following are equivalent:
  \begin{enumerate}
      \item For each $A\in N\cap \mathcal{P}(\eta^{<\omega})$, $j_E(A)\cap \eta^{<\omega}\in N$.
      \item $E\in N$.
  \end{enumerate}
\end{theorem}
\begin{theorem}[Universality Theorem II, Hamkins \cite{HamkinsUniversality}]\label{UniversalityII}
  Suppose  $\delta$ is a regular uncountable cardinal and that $N$ is an inner model of $\mathrm{ZFC}$ containing all the ordinals and having the $\delta$-covering and $\delta$-approximation properties. Suppose  $\gamma>\delta$ is a strong limit cardinal with $\cf(\gamma)\geq \delta$ and that
  $$j\colon V\rightarrow M$$
  is an elementary embedding with $\crit(j)>\delta$ and such that $H(\gamma)\s M$. Letting $E$ the $N$-extender of length $\gamma$ given by $j$ and
  $$j_E\colon N\rightarrow N_E$$
  the corresponding ultrapower, then
  \begin{enumerate}
      \item $N_E\cap H(\gamma)=N\cap H(\gamma)$.
      \item $E\in N$.
  \end{enumerate}
\end{theorem}
Under appropriate assumptions the above theorems apply to $\HOD$:
\begin{theorem}[Woodin, \cite{WooPartI}]\label{thm: WoodinCoveringApprx}
    Assume the $\HOD$ hypothesis holds and that $\delta$ is $\HOD$-supercompact. Then, 
    \begin{enumerate}
        \item $\HOD$ is a weak extender model for the supercompactness of $\delta$.
        \item $\HOD$ has the $\delta$-covering and $\delta$-approximation properties.
    \end{enumerate}
\end{theorem}

\subsection{Forcing prelimminaries}\label{sec: some forcing posets}
For regular cardinals  $\gamma<\mu$ we denote $$S^\mu_{\gamma}:=\{\theta<\mu: \cf(\theta)=\gamma\}.$$ 

A stationary set $S\s S^\mu_\gamma$ is called non-reflecting if $S\cap \alpha$ is non-stationary in $\alpha$ for all $\alpha<\mu.$ There is a natural forcing order to add such an object.
\begin{definition}
Conditions in $\mathbb{NR}(S^\mu_{\gamma})$ are functions $p\colon \alpha\rightarrow \{0,1\}$ for some $\alpha<\mu$ and letting
$S_p=\{\xi<\alpha: p(\xi)=1\},$ $S_p\s S^\mu_\gamma$ and
for every limit $\beta\leq \alpha$ there is a closed unbounded set $c\s \beta$ such that $S_p\cap c=\emptyset.$

Conditions in $\mathbb{NR}(S^\mu_{\gamma})$ are ordered by  reversed inclusion.
\end{definition}

\begin{fact}[Folklore, see {\cite[\S6]{CummingsHandbook}}]Assume $\mu=\mu^{<\mu}$. Then,
    \begin{enumerate}
        \item If $G\s \mathbb{NR}(S^\mu_{\gamma})$ is generic over $V$ then
        $$S_G:=\{\xi<\mu : \exists p\in G,\, p(\xi)=1\}$$
    is a  non-reflecting stationary  such that $S_G\s  S^\mu_{\gamma}$.
        \item $\mathbb{NR}(S^\mu_{\gamma})$ is a $\mu^+$-cc partial order.
        \item $\mathbb{NR}(S^\mu_{\gamma})$ is a $\mu$-strategically-closed partial order.
    \end{enumerate}
\end{fact}




Suppose that $S\s \mu^+$ is a stationary set. There is a standard forcing, due to Abraham and Shelah, to  \emph{shoot} a club through  $S$:

\begin{definition}[Abraham--Shelah, \cite{AbrShe}]
Conditions in $\mathbb{C}(S)$ consists of closed bounded sets $c\s S$.  The order relation of $\mathbb{C}(S)$ is   $\sqsupseteq$-extension.  
\end{definition}

\begin{fact}[{\cite{AbrShe}}]
Suppose that $\mu^{<\mu}=\mu$ and $2^\mu=\mu^+.$
    \begin{enumerate}
        \item $\mathbb{C}(S)$ forces a club $C\s S$.
        \item $\mathbb{C}(S)$ is $\mu^{++}$-cc.
    \end{enumerate}
\end{fact}

The poset $\mathbb{C}(S)$ is in general ill-behaved unless $\mu=\aleph_0$ or $S$ is chosen carefully; for example, $\mathbb{C}(S)$ may collapse $\mu^+$. Abraham and Shelah show in \cite[Theorem~1]{AbrShe} that if $S\s\mu^+$ is \emph{fat} then $\mathbb{C}(S)$ is $\mu^+$-distributive, hence it preserves cardinals ${\geq}\mu^+$.\footnote{A stationary set $S\s\mu^+$ is called fat if there is a club $C\s \mu^+$ such that $C\cap S$ contains closed sets of arbitrarily large order types below $\kappa$.}  
In our particular circumstances we will be able  to ignore this technical aspect thanks to the following fact:



\begin{fact}[Folklore]\label{fact: folklore}
Suppose that $\mu$ is a cardinal, $\gamma\leq \mu$ is regular and that $\dot{S}$ is the canonical $\mathbb{NR}(S^{\mu^+}_{\gamma})$-name for the generic non-reflecting stationary set. Then, \(\textnormal{Add}(\mu^+,1)\) is equivalent to the two step iteration \begin{equation*}
    \mathbb{NR}(S^{\mu^+}_{\gamma}) * \dot{\mathbb{CS}}(\mu^+\setminus\dot{S}).\qedhere
\end{equation*}
    \end{fact}
Given a set  $X\s \alpha$ and a cardinal $\kappa$, there is a standard forcing that \emph{codes} $X$ within the power-set-function-pattern in the interval $[\kappa^+, \kappa^{+\alpha+1}]$. This poset is originally due to McAloon \cite{McAloon} and is defined as follows.

\smallskip

Let $\mathcal{E}_X\colon [\kappa^+,\kappa^{+\alpha}]\rightarrow [\kappa^+,\kappa^{+\alpha}]$ be the partial function defined as
$$\mathcal{E}_X(\kappa^{+\beta+1}):=\begin{cases}
    \kappa^{+\beta+2}, & \text{if $\beta\in X$;}\\
    \kappa^{+\beta+3}, & \text{if $\beta\notin X$}.
    \end{cases}$$
\begin{definition}[Coding]\label{def: codingposet}
    $\mathrm{Code}(X,\kappa)$ is the Easton-product $$\prod_{\beta<\alpha}\mathrm{Add}(\kappa^{+\beta+1}, \mathcal{E}_X(\kappa^{+\beta+1})).$$
\end{definition}
Clearly, $\mathrm{Code}(X,\kappa)$ is a $\kappa^+$-directed-closed forcing. Also, if  the GCH holds in $[\kappa^+,\kappa^{+\alpha}]$, $\mathrm{Code}(X,\kappa)$ codes $X$: For a $V$-generic filter $G\s \mathrm{Code}(X,\kappa)$,
$$X=\{\beta<\alpha: V[G]\models 2^{\kappa^{+\beta+1}}=\kappa^{+\beta+2}\}.$$
Thus,  $X\in \HOD^{V[G]}.$ A classical argument due to McAloon\label{McAloon} \cite{McAloon} shows that $``V=\HOD$'' can be forced using a class iteration of the coding posets $\mathrm{Code}(X,\kappa)$ (see \cite{Reitz}). In fact, this poset forces the stronger axiom denoted $``V=g\HOD$'', where $g\HOD$ stands for the \emph{generic $\HOD$}; namely, $$g\HOD:=\{x: \text{For all set-sized posets $\mathbb{P},\;\one\forces_{\mathbb{P}}\check{x}\in \HOD$}\}.$$

We note that this coding forcing can be arranged to start above any given regular cardinal $\chi$, thereby yielding a $\chi$-directed-closed class product forcing. 

\smallskip

A poset $\mathbb{Q}$ is said to be \emph{cone homogeneous}\label{conehomogeneous} if for each pair of conditions $q_0,q_1\in \mathbb{Q}$ there are $q^*_0\leq q_0$ and $q^*_1\leq q_1$ such that $\mathbb{Q}/q^*_0$ is isomorphic to $\mathbb{Q}/q^*_1$. The next is the key standard fact about cone homogeneous posets.

\begin{fact}[{\cite[Theorem~26.12]{Jech}}]
    If $\mathbb{Q}$ is an ordinal definable and cone-homogeneous forcing then $\HOD^{V[G]}\s V$ for all $V$-generics $G\s \mathbb{Q}$.\qed
\end{fact}
 The above fact will be used both in \S\ref{sec: on the HOD dichotomy} and \S\ref{sec: where weak covering takes a hold} to control the HOD of the relevant generic extensions.

 \smallskip

\begin{definition}\label{def: selfcoding}    Let $\theta$ be an ordinal. A \emph{self-coding iteration} is a sequence $$\langle \mathbb{P}_\alpha, \dot{\mathbb{Q}}_\alpha,
    \dot{A}_\alpha\mid \alpha<\theta\rangle$$ such that the following hold, where $t_\alpha$ denotes the transitive closure of $\mathbb{P}_\alpha$ and $G_\alpha$ denotes the canonical $\mathbb P_\alpha$-name for a generic set:
    \begin{enumerate}
        \item $\dot{\mathbb{Q}}_\alpha$ is a $\mathbb{P}_\alpha$-name for a poset that is forced by the trivial condition  to be $|\mathbb{P}_\beta|$-closed, for all $\beta<\alpha.$
        \item $\dot{A}_\alpha$ is a $\mathbb{P}_\alpha$-name forced by
        the trivial condition to be a binary relation on $|t_\alpha|^V$ whose transitive collapse is $t_\alpha\cup\{\dot{G}_\alpha\}.$
        \item $\mathbb{P}_{\alpha+1}=\mathbb{P}_\alpha\ast \dot{\mathbb{Q}}_\alpha\ast {\mathrm{Code}}(\dot{A}_\alpha,\gamma)$ where $\gamma$ is the first inaccessible above $|\mathbb{P}_\alpha\ast \dot{\mathbb{Q}}_\alpha|$.
        \item At limit stages $\alpha$, the iteration $\mathbb{P}_\alpha$ is constructed as the Easton limit if $\alpha$ is regular, or as the inverse limit if $\alpha$ is singular.
    \end{enumerate}
\end{definition}

\begin{lemma}\label{lemma: self-coding}
    Suppose that $\theta$ is a Mahlo cardinal and let $\vec{\mathbb{P}}=\langle \mathbb{P}_\alpha, \dot{A}_\alpha, \dot{\mathbb{Q}}_\alpha\mid \alpha<\theta\rangle$ be a a self-coding iteration. If $G\s \mathbb{P}_\theta$ is generic over $V$ then $G$ is definable over the structure $(V[G]_\theta, V_\theta, \vec{\mathbb{P}})$ without parameters. 

    In particular, if $V\s \HOD^{V[G]}$ and $\vec{\mathbb{P}}$ is definable over $V_\theta$ then $$\HOD^{V[G]}=V[G].$$
\end{lemma}
\begin{proof}
    For each ordinal $\alpha<\theta$  denote $\pi_\alpha\colon \mathbb{P}_\theta\rightarrow \mathbb{P}_\alpha$  the natural projection and $G_\alpha$ the $\mathbb{P}_\alpha$-generic induced by $G$ and $\pi_\alpha$. 

Since $\theta$ is regular and  $\mathbb{P}_\theta$ is the Easton limit of the previous stages,
$$\textstyle G=\bigcap_{\alpha<\theta}\{p\in\mathbb{P}_\theta\mid \pi_\alpha(p)\in G_\alpha\}.$$
It suffices to show that each $G_\alpha$ is definable over  $(V[G]_\theta, V_\theta, \vec{\mathbb{P}})$ without parameters. 

By the very construction of the iteration, $G_\alpha$ is coded into the continuum pattern of $V[G_{\alpha+1}]$ above  the first inaccessible past $\text{rank}(\mathbb{P}_\alpha\ast\dot{\mathbb{Q}}_\alpha)$. Since the continuum pattern between this cardinal and $\text{rank}(\mathbb{P}_{\alpha+1})$ is the same between $V[G_{\alpha+1}]$ and $V[G]$ it follows that $G_{\alpha}$ is also coded in the continuum pattern of $V[G]$. In turn,  the continuum pattern at these scales is the same as the corresponding one of $V[G]_\theta$. In particular, $G_\alpha$ is definable over the structure $(V[G]_\theta, V_\theta, \vec{\mathbb{P}})$ without parameters. 
\end{proof}

We conclude this section by recalling  the definition of the termspace forcing and proving some basic facts about it. This construction  will be used in the main result of Section~\ref{sec:thefirstsupercompact}.

\begin{definition}\label{def: termspace}
    Let $\mathbb{P}$ be a forcing notion and $\dot{\mathbb{Q}}$ be a $\mathbb{P}$-name for another forcing notion. The \emph{termspace forcing} $\mathbb{A}(\mathbb{P}, \dot{\mathbb{Q}})$ consists of all canonical $\mathbb{P}$-names $\sigma$ such that $\one\forces_{\mathbb{P}}\sigma\in\dot{\mathbb{Q}}$. The ordering of $\mathbb{A}(\mathbb{P}, \dot{\mathbb{Q}})$ is defined as$$\text{$\sigma\leq_{\mathbb{A}(\mathbb{P},\dot{\mathbb{Q}})}\tau$ if and only if
    $\one\forces_{\mathbb{P}}\sigma\leq_{\dot{\mathbb{Q}}} \tau.$}$$
\end{definition}

\begin{lemma}\label{lemma: facts termspace}\hfill
    \begin{enumerate}
        \item If $\one\forces_{\mathbb{P}} ``\dot{\mathbb{Q}}$ is $\kappa$-closed" then $\mathbb{A}(\mathbb{P},\dot{\mathbb{Q}})$ is $\kappa$-closed.
        \item If $G\s \mathbb{P}$ and $H\s \mathbb{A}(\mathbb{P}, \dot{\mathbb{Q}})$ are generic filters over $V$. Then the set $\{\tau_G\mid \tau\in H\}$ is a generic filter for $\dot{\mathbb{Q}}_G$ over $V[G]$.
        \item Suppose that $\kappa$ is a weakly compact. If $|\mathbb{P}|<\kappa$ and $\one\forces_{\mathbb{P}}``\dot{\mathbb{Q}}$ is $\kappa$-cc" then $\mathbb{A}(\mathbb{P},\dot{\mathbb{Q}})$ is $\kappa$-cc.
    \end{enumerate}
\end{lemma}
\begin{proof}
    (1) is a straightforward application of the \emph{Forcing Maximality Principle.} (2) is standard (see \cite[Proposition~22.3]{CummingsHandbook}). For (3), suppose that $\{\sigma_\alpha\mid \alpha<\kappa\}$ was an antichain in $\mathbb{A}(\mathbb{P},\dot{\mathbb{Q}})$. Thus, for each pair $\alpha<\beta$ there is a condition $p_{\alpha,\beta}\forces_{\mathbb{P}}\sigma_\alpha\perp_{\dot{\mathbb{Q}}} \sigma_\beta$. Let $c\colon [\kappa]^2\rightarrow\mathbb{P}$ given by $\langle\alpha,\beta\rangle\mapsto p_{\alpha,\beta}$.  Since $|\mathbb{P}|<\kappa$ and $\kappa$ is weakly compact there is an unbounded set $H\s \kappa$ and $p^*\in \mathbb{P}$ such that $p_{\alpha,\beta}=p^*$ for all $\langle\alpha,\beta\rangle\in H$. Thus $p^*\forces_{\mathbb{P}}``\dot{\mathbb{Q}}$ is not $\kappa$-cc", contrary to our assumption.
\end{proof}

\section{The HOD Dichotomy and extendible-like cardinals}\label{sec: on the HOD dichotomy}

This section consists of three subsections: \S\ref{SectionFirstExtendible},  \S\ref{SectionTransferring} and \S\ref{sec: CCE}. In \S\ref{SectionFirstExtendible} we prove Theorem~\ref{thm:HOD_thm} which says that (even under the HOD hypothesis) the first extendible cardinal can be the first strongly compact cardinal in $\HOD$.  In \S\ref{SectionTransferring} we complement this result by showing that, again under the $\HOD$ hypothesis, the first extendible cardinal must be $C^{(1)}$-supercompact in $\HOD$, therefore a  ``big" supercompact cardinal in $\HOD$. 
Finally, in \S\ref{sec: CCE} we discuss \emph{cardinal-correct extendible} cardinals and show that in $\HOD$ of the model of \S\ref{SectionFirstExtendible} the first cardinal-correct extendible is not extendible in $\HOD$. 

\subsection{The first extendible can be the first strongly compact of $\HOD$}\label{SectionFirstExtendible}
The basic idea is to prove this by forcing with an iteration that shoots non-reflecting stationary sets and immediately after kills them. Our arguments combine ideas from  \cite{GoldbergOptimality} and Cheng--Friedman--Hamkins \cite{CFH}. We will also employ the characterization of extendibility via \emph{reflecting measures} due to Bagaria and Goldberg. (See Theorem~\ref{CharacterizingExtendibility} above.)

\smallskip


\begin{theorem}\label{thm:HOD_thm}
   Let $\delta$ be an extendible cardinal. Then there is a generic extension of the universe of sets where the \(\HOD\) hypothesis holds, $\delta$ remains extendible, and $\delta$ is the least strongly compact cardinal of $\HOD$.
    \end{theorem}
    \begin{proof}
   By virtue of \cite[Theorem~3.8 and Proposition~3.9]{GoldbergOptimality},
        we may start with a model $V$ where the following properties hold:
        \begin{enumerate}
            \item $\delta$ is an  extendible cardinal.
             \item  $\Sigma_2$-correct cardinals are  $\Sigma_2$-indestructibly correct. 
            
        \end{enumerate}
\begin{claim}
     For each ordinal $\alpha<\delta$, there is an interval of cardinals $I\s \delta$ with $\otp(I)=\alpha$ such that $2^\theta=\theta^+$ for all $\theta\in I.$ Also $V=\mathrm{gHOD}$ holds.
\end{claim}
\begin{proof}[Proof of claim]
 Fix an ordinal $\alpha<\delta$ and let $\lambda>\delta$ be a $\Sigma_2$-indestructibly correct  cardinal. We may force,  using $\lambda$-directed-closed forcing, the GCH within an interval of cardinals of order-type $\alpha$ above $\lambda$. In particular, we can force this configuration while keeping $\lambda$ a $\Sigma_2$-correct cardinal. By $\Sigma_2$-correctness of $\lambda$ in the generic extension there must be an interval of cardinals below $\lambda$ of order-type $\alpha$ where the GCH holds. But this poset does not change $V_\lambda$, so the said interval of cardinals  existed back in the ground model $V$. Finally, since $\delta$ is extendible (hence $\Sigma_3$-correct) in $V$ and $\alpha<\delta$, an interval of cardinals $I$ as in the claim must exist below $\delta$. 

 The argument for $V=\mathrm{gHOD}$ is similar.
\end{proof} 
Notice that the reflection argument presented above allows for the selection of an interval $I \subseteq \delta$ with a minimal element greater than any fixed $\beta < \delta$.

\smallskip


        We will make $\delta$ become the first strongly compact in $\HOD^{V[G]}$ by adding many non-reflecting stationary sets below $\delta$. Since  we wish to 
   keep $\delta$ extendible in $V[G]$, we plan to kill these stationary sets right after they are introduced employing the club shooting poset. By Fact~\ref{fact: folklore}, each of these two step iterations will be equivalent to Cohen forcing, and as a result they will be nicely behaved. The key point is that the  non-reflecting stationaries will be coded into $\HOD^{V[G]}$ while their club counterparts will  not. 

   \smallskip

   Let us start by fixing an \emph{extendible Laver function} $f\colon \delta\rightarrow\delta$. Specifically, $f$ has the following guessing property: For each $\Sigma_2$-indestructibly correct cardinal $\lambda>\delta$ there is a normal fine $\delta$-complete ultrafilter $\mathcal{U}$ on $T_{\delta,\lambda}$ such that $j_{\mathcal{U}}(f)(\delta)=\lambda$. Such a function exist as per the argument in \cite[Theorem~3.10, p.10]{GoldbergOptimality}. Tsaprounis previously observed in \cite{Tsan} that extendible-like Laver functions, similar to those considered here, exist.
   
   \smallskip

   Let $\mathbb{P}$ be the Easton-supported product $\prod_{\alpha<\delta}\mathbb{Q}_\alpha$ where $\mathbb{Q}_\alpha$ is defined recursively follows: Suppose that $\langle \mathbb{Q}_\beta: \beta<\alpha\rangle$ has been defined and that $|\mathbb{Q}_\beta|<\delta$ for all $\beta<\alpha$. If $\alpha$ is an inaccessible cardinal with $f[\alpha]\s \alpha$, 
 then we take $\lambda_\alpha$ the least cardinal above $\max(f(\alpha), \sup_{\beta<\alpha}\lambda_\beta)$ and stipulate that $$\mathbb{Q}_\alpha:=\mathbb{NR}(S^{\lambda_\alpha^+}_\omega)\ast {\dot{\mathbb{CS}}(\lambda_\alpha^+\setminus\dot{S}_\alpha)\ast \dot{\mathrm{Code}}(\dot{S}_\alpha,\min(I_\alpha))},$$
   where $I_\alpha\s\delta$ is an interval of cardinals above $\lambda_\alpha^+$ with order-type $\lambda_\alpha^+$ such that $2^\theta=\theta^+$ for all $\theta\in I$. (This choice of $I_\alpha$ guarantees that forcing with $\mathrm{Code}(\dot{S}_\alpha,\min(I_\alpha))$ codes $\dot{S}_\alpha$ into the power-set function above $\min(I_\alpha)$.) 

   \smallskip

   In any other circumstances $\mathbb{Q}_\alpha$ is declared to be the trivial and $\lambda_\alpha:=0$. 

\smallskip

Note that $\{\mathbb{Q}_\alpha: \alpha<\delta\}\s V_\delta$. Since we are taking Easton supports this implies that $\mathbb{P}\s V_\delta.$

   \smallskip

   To streamline the forthcoming discussion  we shall denote:

   \begin{itemize}
       \item $\mathbb{NR}_{\alpha}:=\mathbb{NR}(S^{\lambda^+_\alpha}_\omega);$
       \item $\dot{\mathbb{CS}}_\alpha:=\mathbb{CS}(\lambda^+_\alpha\setminus \dot{S}_\alpha)$;
       \item $\dot{\mathrm{Code}}_\alpha:=\dot{\mathrm{Code}}(\dot{S}_\alpha,\min(I_\alpha))$.
   \end{itemize}

   In what follows $G\s \mathbb{P}$ is a fixed $V$-generic filter.

   \begin{claim}
      $\delta$ remains extendible in $V[G].$ 
   \end{claim}
   \begin{proof}
      We use Theorem~\ref{CharacterizingExtendibility}---that is, the Bagaria--Goldberg ultrafilter-based characterization of extendibility. Let $\lambda>\delta$ be a $\Sigma_2$-indestructibly  correct cardinal such that $\cf(\lambda)>\delta$. (Note that such a cardinal exist by our departing assumption.). By Solovay's theorem the SCH holds at $\lambda$, hence   $\lambda^{<\delta}=\lambda$. 
Let $\mathcal{U}$ be a normal fine $\delta$-complete ultrafilter on $T^*_{\delta,\lambda}$ such that $j_{\mathcal{U}}(f)(\delta)=\lambda$ where $j_{\mathcal{U}}\colon V\rightarrow M_\mathcal{U}$
denotes the corresponding ultrapower embedding.


       


Thus one obtains the following factoring:
$$\textstyle j_\mathcal{U}(\mathbb{P})\simeq \mathbb{P}\times {\mathbb{Q}}=\mathbb{P}\times  (\prod_{\delta\leq \alpha<j_{\mathcal{U}}(\delta)}\mathbb{Q}_\alpha)^{M_{\mathcal{U}}}.$$
Let us take a closer look at the tail forcing $\mathbb{Q}$. Since $\delta$ is inaccessible and a closure point of $j_{\mathcal{U}}(f)$ (i.e., $j_{\mathcal{U}}(f)[\delta]=f[\delta]\s \delta$) the definition of $j_{\mathcal{U}}(\mathbb{P})$ in $M_{\mathcal{U}}$ guarantees that $\mathbb{Q}_\delta$ is a non-trivial partial order. Thus, 
$$(\mathbb{Q}_\delta)^{M_{\mathcal{U}}}=(\mathbb{NR}_\delta\ast \dot{\mathbb{CS}}_\delta\ast \dot{\mathrm{Code}}_\delta)^{M_{\mathcal{U}}}.$$
By our choice, $(\lambda_\delta)^{M_\mathcal{U}}>j_{\mathcal{U}}(f)(\delta)=\lambda$, so that $(\mathbb{NR}_\delta\ast \dot{\mathbb{CS}}_\delta)^{M_{\mathcal{U}}}$ contains in $M_{\mathcal{U}}$ a dense $(\lambda^+)^{M_{\mathcal{U}}}$-directed-closed forcing, $\mathbb{D}_\delta$ (see Fact~\ref{fact: folklore}). Since $M_{\mathcal{U}}$ is closed under $\lambda$-sequences  it follows that $\mathbb{D}_\delta$ is  $\lambda^+$-directed-closed   in $V$.

The same argument shows that the tail forcing $\mathbb{Q}$ has a dense subset of conditions that is $\lambda^+$-directed-closed in $V$. The reason is that $(\lambda_\alpha)^{M_{\mathcal{U}}}>(\lambda_\delta)^{M_{\mathcal{U}}}>j_{\mathcal{U}}(f)(\delta)^+=\lambda^+$ for all $\alpha\in (\delta,j_{\mathcal{U}}(\delta))$ where $(\mathbb{Q}_\alpha)^{M_{\mathcal{U}}}$ is non-trivial. 

\smallskip





Notice that any $V$-generic for $$\textstyle \mathbb{C}:=\prod_{\delta\leq \alpha<j_{\mathcal{U}}(\delta)}\mathbb{D}_\alpha\ast (\dot{\mathrm{Code}}_ \alpha)^{M_{\mathcal{U}}}$$ induces (by density) a $V$-generic filter for $\mathbb{Q}$. Since $\mathbb{C}$ is a $\lambda^+$-directed-closed forcing in $V$ and there are at most $\lambda^{+}$-many dense open sets for it in $M_{\mathcal{U}}$\footnote{Here we use that $\lambda^{<\delta}=\lambda$, $2^\lambda=\lambda^+$ and the usual extender representation of dense open subsets  $D\s \mathbb{Q}$ in $M_{\mathcal{U}}$ (i.e.,  $D=j(f)(j``\lambda)$ for a function $f\colon \mathcal{P}_\delta(\lambda)\rightarrow\mathcal{P}(\mathbb{P})$).} we can construct (inside $V$) a $V$-generic filter for $\mathbb{C}$, hence also a $V$-generic filter $H\s \mathbb{Q}$. Moreover, since $|\mathbb{P}|<\lambda$, a mutual genericity argument shows that any $M_{\mathcal{U}}$-generic  for $\mathbb{Q}$ is in fact $M_{\mathcal{U}}[G]$-generic. Thereby  we obtain a $M_{\mathcal{U}}[G]$-generic filter $H\in V[G]$ for the poset $\mathbb{Q}$.

\smallskip

At this stage we are almost done with the preservation of extendibility. Thanks to the Easton support $j_{\mathcal{U}}`` G\s G\times H$, hence $j_{\mathcal{U}}\colon V\rightarrow M_\mathcal{U}$ lifts to
$$j_{\mathcal{U}}\colon V[G]\rightarrow M_\mathcal{U}[G\times H]=M_\mathcal{U}[H\times G].$$
Since $\lambda$ was $\Sigma_2$-indestructibly  correct in $M_\mathcal{U}$ it remains $\Sigma_2$-correct in $M_\mathcal{U}[H]$. Also, $G$ is a generic for a small forcing relative to $\lambda$, so the $\Sigma_2$-correctness of the cardinal is preserved in $M_\mathcal{U}[H\times G]$. To conclude define $$\mathcal{W}:=\{x\in \mathcal{P}_\delta(\lambda)^{V[G]}: j``\lambda\in j_{\mathcal{U}}(X)\}$$
which is easily shown to be a $2$-reflecting measure on $\mathcal{P}_\delta(\lambda)$.

This completes the verification that $\delta$ is extendible in $V[G]$.
   \end{proof}

Let us now argue that $\delta$ is the first strong compact cardinal in \(\HOD^{V[G]}\). 

\smallskip

Denote    
$\textstyle \pi \colon \mathbb{P}\rightarrow\prod_{\alpha<\delta} \mathbb{NR}_\alpha\;\;
\text{and}\;\;
\textstyle \pi_\alpha\colon \prod_{\alpha<\delta} \mathbb{NR}_\alpha\rightarrow \mathbb{NR}_\alpha$
 the natural projections. Let $G_0$ be the $V$-generic induced by $G$ and $\pi$, and $G_{0,\alpha}$ the $V$-generic induced by $\pi_{\alpha}$ and $G_0$, respectively.  The next claim is key:

 \begin{claim}
     \(\HOD^{V[G]}=V[G_0]\).
 \end{claim}
 \begin{proof}[Proof of claim]
    To prove the more difficult inclusion, that \(\HOD^{V[G]}\subseteq V[G_0]\), it suffices to show that in $V[G_0]$, $\mathbb{P}/G_0$ is cone homogeneous and definable from \(G_0\) and parameters in \(V\).
    It is tempting to say that the quotient forcing $\mathbb{P}/G_0$ is the Easton product $\prod_{\alpha<\delta}(\mathbb{CS}_{\alpha}\ast \dot{\mathrm{Code}}_{\alpha})_{G_{0,\alpha}}$ as computed in \(V[G_0]\), but we do not know whether this is true. 
    Instead, we show the following:
    \begin{lemma}
        Suppose \(\mathbb P\) is the Easton product \(\prod_{\alpha < \delta} \mathbb A_\alpha * \dot{\mathbb B}_\alpha\)
        and for each \(\alpha < \delta\), \(\mathbb A_\alpha\) forces that \(\dot{\mathbb B}_\alpha\) is cone homogeneous. Then if \(H\subseteq \prod_{\alpha < \delta}\mathbb A_\alpha\) is \(V\)-generic, \(\mathbb P/H\) is cone homogeneous.
        \begin{proof}
            Let \(H_\alpha\subseteq \mathbb A_\alpha\) be the generic filter induced by \(H\). Let 
            \(\mathbb B_\alpha = (\dot{\mathbb B}_\alpha)_{H_\alpha}\).
            Working in \(V[H]\), say a condition \((b_\alpha)_{\alpha < \delta}\in \prod_{\alpha < \delta}\mathbb B_\alpha\) is \textit{relevant} if there is an Easton-supported sequence
            \((\dot b_{\alpha})_{\alpha < \delta}\in (\prod_{\alpha < \delta} \dot{\mathbb B}_\alpha)^V\) such that
            \(b_\alpha = (\dot{b}_\alpha)_{H_\alpha}\) for each \(\alpha < \delta\). An easy calculation shows that
            the separative quotient of
            \(\mathbb P/H\) is isomorphic to the suborder \(\mathbb B\subseteq \prod_{\alpha < \delta} \mathbb B_\alpha\) consisting of relevant conditions.\footnote{
            Recall here that \(\mathbb P/H\) is equal to the set of conditions  \( (p_\alpha,\dot{q}_\alpha)_{\alpha < \delta}\in \mathbb P
            \) such that \((p_\alpha)_{\alpha < \delta}\in H\). The forcing equivalence of \(\mathbb P/H\) and \(\mathbb B\) is witnessed by the map $f : \mathbb P/H\to \mathbb B$ defined by
            \(f((p_\alpha,\dot{q}_\alpha)_{\alpha < \delta}) = ((\dot{q}_\alpha)_{G_\alpha})_{\alpha < \delta}
            \), and moreover, \(f\) is the separative quotient map.
            } Given a condition $b\in \mathbb{B}$ we denote by $\mathbb{B}_b$ the subposet of $\mathbb{B}$ whose conditions are all stronger than $b$.

            Suppose \((b_\alpha)_{\alpha < \delta}\)
            and \((c_\alpha)_{\alpha < \delta}\) are relevant conditions, and we will define, in \(V[H]\), relevant conditions \((b^*_\alpha)_{\alpha < \delta}\leq (b_\alpha)_{\alpha < \delta}\)
            and \((c^*_\alpha)_{\alpha < \delta}\leq (c_\alpha)_{\alpha < \delta}\) and
            an isomorphism \(\pi : \mathbb B_{b^*}\to \mathbb B_{c^*}\).
            Let \((\dot{b}_\alpha)_{\alpha < \delta}\)
            and \((\dot{c}_\alpha)_{\alpha < \delta}\) witness the relevance of \((b_\alpha)_{\alpha < \delta}\)
            and \((c_\alpha)_{\alpha < \delta}\). Working in \(V\), for each
            \(\alpha < \delta\) in the union of the supports of \((\dot{b}_\alpha)_{\alpha < \delta}\)
            and \((\dot{c}_\alpha)_{\alpha < \delta}\), fix \(\mathbb A_\alpha\)-names \(\dot{b}_\alpha^*\), 
            \(\dot{c}_\alpha^*\), and \(\dot{\pi}_\alpha\)
            such that \(\mathbb A_\alpha\) forces that
            \(\dot{b}_\alpha^* \leq \dot b_\alpha\),
            \(\dot{c}_\alpha^* \leq \dot c_\alpha\),
            and \(\dot{\pi}_\alpha^*\colon (\dot{\mathbb B}_\alpha)_{\dot b_\alpha^*}\to (\dot{\mathbb B}_\alpha)_{\dot c_\alpha^*}\) is an isomorphism.
            
            Let \(\pi_\alpha = (\dot{\pi}_\alpha)_{H_\alpha}\),
            \(b^*_\alpha = (\dot{b}_\alpha^*)_{H_\alpha}\), and
            \(c^*_\alpha = (\dot{c}_\alpha^*)_{H_\alpha}\).
            Note that \(b^*\) and \(c^*\) are relevant conditions.
            Define \(\pi : \mathbb B_{b^*}\to \mathbb B_{c^*}\) by
            \(\pi((u_\alpha)_{\alpha < \delta}) = (\pi_\alpha(u_\alpha))_{\alpha < \delta}\). Then \(\pi\) is the desired isomorphism. 
        \end{proof}
    \end{lemma}
    
    Clearly, \((\mathbb{CS}_{\alpha}\ast \dot{\mathrm{Code}}_{\alpha})_{G_{0,\alpha}}\) is a cone homogeneous forcing in $V[G_{0,\alpha}]$, since in \(V[G_{0,\alpha}]\), finite iterations of $\text{OD}_{G_{0,\alpha}}$ cone homogeneous forcings are cone homogeneous; this is a relativized version of \cite[Lemma~8.3]{BenHay}, which in turn follows from results of \cite{DobFri}. Therefore we can apply the lemma above to obtain that the quotient \(\mathbb P/G_0\) is cone homogeneous. It follows that $$\HOD^{V[G]}\s V[G_0].$$
    
 \smallskip
 
 Furthermore, we claim that $G_0\in \HOD^{V[G]}$ and so 
 $$\HOD^{V[G]}=V[G_0].$$
This is because \(G\) was constructed so as to code all the non-reflecting stationary sets into the continuum function. 
Specifically, $G_0$ is ordinal definable in $V[G]$ taking $\prod_{\alpha<\delta}\mathbb{NR}_\alpha$ and $\langle \lambda_\alpha: \alpha<\delta\rangle$ as parameters.
Since $V\models ``V=\mathrm{gHOD}$'' (in fact, $V$ satisfies Reitz's \emph{Continuum Coding Axiom} \cite{Reitz}, which implies $V=g\HOD$) both of these parameters belong to $\HOD^{V[G]}$
 which yields $G_0\in \HOD^{V[G]}$. 
 \end{proof}

The $\HOD$ hypothesis holds in $V[G]$ because $\HOD^{V[G]}$ is a generic extension of $\HOD^V$ by a set-size forcing. As a result, $\HOD^{V[G]}$ is a weak extender model for the supercompactness of $\delta$  and therefore $\delta$ is supercompact in this model. Finally,  $\delta$ must be the first strongly compact cardinal in $\HOD^{V[G]}$ precisely because $\HOD^{V[G]}=V[G_0]$.
\end{proof}

\begin{remark}
 Employing standard techniques one can also arrange the \emph{mirror image configuration} of  Theorem~\ref{thm:HOD_thm}. Namely, one can prove the consistency (with the HOD hypothesis) of the first strongly compact cardinal being extendible in HOD. The idea is simple: Start with a model for $V=\mathrm{gHOD}$ accomodating an extendible cardinal $\delta$ (see \cite{BagPov}). Force with a \emph{Magidor product} of Prikry forcings as in \cite{HMP}. This forcing preserves the strong compactness of $\delta$ 
and it is not hard to check that it is cone homogeneous. As a result, in the generic extension $V[G]$, $\delta$ becomes the first strongly compact while remains extendible (non-necessarily the first one!) in $\HOD^{V[G]}$.
\end{remark}

\begin{remark}

Note that the GCH fails in the model $V[G]$ from the previous construction, which raises the question of whether this failure is necessary. The answer is no. One can achieve the configuration of Theorem~\ref{thm:HOD_thm} while preserving the GCH as follows: Start with a model of GCH and prepare it so that every inaccessible $\Sigma_2$-correct cardinal $\lambda$ becomes  indestructible under $\lambda$-directed-closed forcing that preserve the GCH. Lemma~\ref{GCHplussigma2correctness} below shows that this preparation is feasible, and that it preserves both the GCH and extendible cardinals. In the resulting model, instead of coding the stationary sets $S_\alpha$ into the power-set-function pattern, use $\diamondsuit^*$ as in \cite{Brooke-Taylor} for the coding. The charm of this alternate coding is that it preserves the GCH. Finally force with the modified Easton support product $\mathbb{P}$.
\end{remark}
\begin{lemma}[GCH-preserving preparation]\label{GCHplussigma2correctness}
  There is a class forcing $\mathbb{Q}$ such that every  $\Sigma_2$-correct cardinal  becomes  $\Sigma_2$-indestructibly correct under posets that preserve the $\mathrm{GCH}$. Moreover, this forcing preserves extendible cardinals and if the $\mathrm{GCH}$ holds in $V$, then so does it in $V^{\mathbb{Q}}$.
\end{lemma}
\begin{proof}
    The proof is a variation of the argument preceding \cite[Theorem~3.7]{GoldbergOptimality} aimed at guaranteeing that the GCH is preserved. For an ordinal $\eta$, a set $y$ of rank at most $\eta$ and a $\Sigma_2$ formula in the language of set theory $\varphi(x)$, we denote by $\Phi(\eta,y,\varphi)$ the sentence  ``There is an $\eta$-directed-closed forcing $\mathbb{P}$ that preserves the GCH  such that for some $\beta\in \ord$ and $p\in \mathbb{P}$, $p\forces_{\mathbb{P}}V[\dot{G}]_\beta\models \varphi(\check{y})$''.    If $\Phi(\eta, y,\varphi)$ holds we let $F(\eta, y, \varphi)$ be the first $\beth$-fixed point $\beta$ such that $\Phi(\eta, y,\varphi)$ is witnessed by $\beta$ and some $\mathbb{P}\in V_\beta$. Otherwise, if $\Phi(\eta, y,\varphi)$ fails we declare $F(\eta, y, \varphi):=0$. Set $$\textstyle f(\eta):=\sup_{y\in V_\eta, \varphi\in \Sigma_2}F(\eta,y,\varphi).$$

\smallskip

   We define an Easton-supported class iteration $$\mathbb{Q}:=\varinjlim\langle \mathbb{Q}_\alpha, \dot{\mathbb{P}}_\alpha: \alpha\in \ord\rangle$$
   together with a countinuous increasing sequence of ordinals $\langle \eta_\alpha: \alpha\in \ord\rangle$ as follows: Suppose that  both $\mathbb{Q}_\alpha$ and $\eta_\alpha$ have been defined. Then we let $\dot{\mathbb{P}}_\alpha$ be a $\mathbb{Q}_\alpha$-name for {the lottery sum} (see \cite[\S3]{HamLot})  of all the GCH-preserving $\eta_\alpha$-directed-closed forcings in $V_{f(\eta_\alpha)}$ (this computation takes place in $V^{\mathbb{Q}_\alpha}$).

   \smallskip

   The argument of \cite[Theorem~3.8]{GoldbergOptimality} shows that forcing with $\mathbb{Q}$ makes every $\Sigma_2$-correct cardinal of $V$ become indestructibly $\Sigma_2$-correct under posets preserving the GCH. The same conclusion applies to the preservation of extendible cardinal from the ground model $V$ \cite[Proposition~3.9]{GoldbergOptimality}.  

   \begin{claim}
       If the $\mathrm{GCH}$ holds in $V$ then so does in $V^{\mathbb{Q}}$.
   \end{claim}
   \begin{proof}[Proof of claim]
       Let $G\s \mathbb{Q}$  a $V$-generic filter and fix $\lambda$ a cardinal in $V[G]$. Since $\langle \eta_\alpha: \alpha\in \ord\rangle$ is an increasing continuous sequence  there is an ordinal $\alpha$ such that $\eta_\alpha\leq \lambda<\eta_{\alpha+1}$. By the closure properties of  $\mathbb{Q}/G_{\alpha+1}$, $$\mathcal{P}(\lambda)^{V[G]}=\mathcal{P}(\lambda)^{V[G_{\alpha+1}]}.$$
       By the construction of $\mathbb{Q}$ it follows that $|\mathbb{Q}_\alpha|\leq \eta_\alpha$. Standard   nice name arguments together with $``V\models\mathrm{GCH}"$  yield $V[G\restriction\alpha]\models 2^\lambda=\lambda^+$. By assumption, $(\dot{\mathbb{P}}_\alpha)_{G_\alpha}$ preserves the GCH so $V[G]\models 2^\lambda=\lambda^+$. 
   \end{proof}
   The above finishes the proof of the proposition.
\end{proof}





Let us conclude this section with an interesting corollary of Theorem~\ref{thm:HOD_thm}.  
\begin{cor}\label{cor: identity crises}
The following are consistent with the $\HOD$ hypothesis:
\begin{enumerate}
    \item The first extendible is the first $\HOD$-supercompact.
    \item $C^{(1)}$-supercompact cardinals may not be strongly compact in $\HOD$.
\end{enumerate}

\end{cor}
\begin{proof}
Let $V[G]$ be the model constructed in Theorem~\ref{thm:HOD_thm}.

(1) Since $\delta$ is extendible in $V[G]$ it is $\HOD^{V[G]}$-supercompact by \cite[Lemma 188]{WooPartI}. Note that there are no $\HOD^{V[G]}$-supercompact cardinals $\kappa<\delta$ as those $\kappa$ would be supercompact in $\HOD$, (here we used the HOD hypothesis  in $V[G]$ and  Theorem~\ref{thm:WoodinHODsupercompact})  but $\delta$ is the first supercompact in $\HOD$.

\smallskip

(2) Poveda showed in \cite{PovOmega} that every extendible cardinal is a limit of $C^{(1)}$-supercompacts. Since the first extendible in the model of Theorem~\ref{thm:HOD_thm} is the first strongly compact in HOD, many $C^{(1)}$-supercompact cardinals (in $V[G]$) are not strongly compacts in $\HOD$.
\end{proof}

\subsection{The first supercompact may not be supercompact in $\HOD$} \label{sec:thefirstsupercompact}

The goal of this section is to show that even under the $\HOD$ hypothesis the first supercompact cardinal may be strongly compact but not supercompact in $\HOD$. This answers a question of Cheng--Hamkins--Friedman from 
\cite{CFH}.

\begin{theorem}\label{theorem:main}
   Suppose that $\kappa$ is supercompact. Then, there is a generic extension of the set-theoretic universe where the following hold:
    \begin{enumerate}
        \item The $\HOD$ Hypothesis.
         \item $\kappa$ is supercompact.
         \item $\HOD\models ``\kappa$ is strongly compact yet not $2^\kappa$-supercompact".
    \end{enumerate}
  If $\kappa$ is the first supercompact then it remains so in the above extension.
\end{theorem}
The proof is motivated by the following fact proven in \cite{Gol}:
\begin{theorem}[$\HOD$ Hypothesis]\label{thm:hod_closure}
    Suppose that \(\kappa\) is supercompact. Then 
    there is a weak extender model $M$ of $\kappa$ is supercompact that contains
     \(\HOD\) as a ${<}\kappa$-closed inner model. 
\end{theorem}
By Theorem \ref{thm:hod_closure}, if one wants to produce a model as in Theorem \ref{theorem:main}, one must at least be able to produce a model \(H\) containing a strongly compact cardinal \(\kappa\) that is not supercompact in $H$ but becomes supercompact in a \(\kappa\)-distributive forcing extension \(M\) of $H$. The proof of Theorem \ref{theorem:main} produces such a model \(H\) by forcing over $V$ adding one Cohen subset to $\delta$ for every measurable cardinal \(\delta\) below a $V$-supercompact cardinal \(\kappa\). The supercompactness of \(\kappa\) can be restored by passing to the generic extension \(M\) of \(H\) obtained adding one Cohen subset of $\kappa$; the strong compactness of $\kappa$ is preserved using a termspace forcing argument due to Magidor. 
Finally, by preparing the ground model $V$ and using a self-coding iteration  in the forcing to produce $H$, we will ensure that \(H = \text{HOD}^M\).

\begin{proof}[Proof of Theorem \ref{theorem:main}]
Without loss of generality assume that there are no inaccessible cardinals above $\kappa$. (This can be achieved by truncating the universe below the first inaccessible cardinal above \(\kappa\), if there were any). By passing to a further class forcing extension of this model, we may assume that $\kappa$ is supercompact,  $V=\mathrm{gHOD}$, and the GCH holds. These assumptions will be useful in the subsequent analysis of HOD of the final generic extension.

    \smallskip

    Let $\mathbb{P}_\kappa$ be the direct limit of the self-coding iteration 
    $\langle \mathbb{P}_\alpha, \dot{A}_\alpha, \dot{\mathbb{Q}}_\alpha\mid \alpha<\kappa\rangle$ where for each $\alpha<\kappa$, $\dot{\mathbb{Q}}_\alpha$ is a $\mathbb{P}_\alpha$-name for either $\add(\alpha,1)$ if $\alpha$ is a measurable cardinal, or  the trivial forcing otherwise.
    

    \smallskip

    Let $G\s\mathbb{P}_\kappa$ a generic filter.

    \begin{claim}
        The following hold in $V[G]$:
        \begin{enumerate}
            \item $\kappa$ is strongly compact.
            \item $\kappa$ is not $2^\kappa$-supercompact.
        \end{enumerate}
    \end{claim}
    \begin{proof}[Proof of claim]

    (1) The argument uses an idea due to Magidor to make the first strongly compact cardinal the first measurable \cite[\S22]{CummingsHandbook}.

    \smallskip

    Let $\lambda>2^\kappa$ and $j\colon V\rightarrow M$ be a $\lambda$-supercompact embedding. Since $\kappa$ is measurable in $M$ we can let $\mathcal{U}\in M$ a normal measure on $\kappa$. Furthermore, we choose a minimal $\mathcal{U}$  in the Mitchell order on $\kappa$. Let $N$ be the transitive collapse of the $M$-ultrapower by $\mathcal
    U$. By minimality of the measure,  $$\text{$N\models ``\kappa$ is not measurable".}$$
Thus we have the following commutative system of embeddings:
 \begin{displaymath}
       \begin{tikzcd}
  V \arrow{r}{j} \arrow[bend right]{rr}{k} & M \arrow{r}{i}  & N
\end{tikzcd}
    \end{displaymath}
Note that $k(\kappa)=j(\kappa)$ as $j(\kappa)$ is an $M$-inaccessible, hence a fixed point of $i$.

\smallskip

We also have the following factorisation of $j(\mathbb{P}_\kappa)$ and $k(\mathbb{P}_\kappa):$
\begin{itemize}
    \item $j(\mathbb{P}_\kappa)\simeq \mathbb{P}_\kappa\ast (\dot{\mathbb{Q}}_\kappa)^M\ast (\dot{\mathbb{P}}_{(\kappa,j(\kappa))})^M$.
      \item $k(\mathbb{P}_\kappa)\simeq \mathbb{P}_\kappa\ast \{\one\}\ast  (\dot{\mathbb{P}}_{(\kappa,i(\kappa))})^N\ast (\dot{\mathbb{Q}}_{i(\kappa)})^N\ast (\dot{\mathbb{P}}_{(i(\kappa),j(\kappa))})^N$.
\end{itemize}
A few data points in regards to the above factorisation:
\begin{enumerate}
  \item At stage $\kappa$ the iteration $k(\mathbb{P}_\kappa)$ is trivial as $\kappa$ is not measurable in $N.$
    \item Working in $M[G]$, $(\dot{\mathbb{Q}}_\kappa)^M_G\ast (\dot{\mathbb{P}}_{(\kappa,j(\kappa))})^M$ is isomorphic to
    $$\mathrm{Add}(\kappa,1)_{M[G]}\ast (\mathrm{Code}(\dot{A}_\kappa, \iota_\kappa^M)\ast (\dot{\mathbb{P}}_{(\kappa,j(\kappa))})^M).$$
    Denote the second of these iterates by $\dot{\mathbb{T}}^M$.

    Note that $\mathrm{Add}(\kappa,1)_{M[G]}$ forces:
   $$ \text{$``\dot{\mathbb{T}}^M$ is  $j(\kappa)$-cc and   $\lambda^+$-closed".}$$
    For the second of these properties note that the first $M$-inaccessible above $\text{rank}({\mathbb{P}_\kappa\ast \dot{\add}(\kappa,\kappa^{++})})$ (i.e., $\iota_\kappa^M$) is above $\lambda$ by virtue of our assumption   about the non-existence of inaccessible cardinals past $\kappa$. 
     \item $(\mathbb{P}_{i(\kappa)})^N$ forces $(\dot{\mathbb{Q}}_{i(\kappa)})^N\ast (\dot{\mathbb{P}}_{(i(\kappa),j(\kappa))})^N$ to be isomorphic to
    $$\mathrm{Add}(i(\kappa),1)^N\ast (\mathrm{Code}(\dot{A}_{i(\kappa)}, \iota_{i(\kappa)}^N)\ast (\dot{\mathbb{P}}_{(i(\kappa),j(\kappa))})^N).$$
    Denote the second of these iterates by $\dot{\mathbb{T}}^N$. Note that $i(\dot{\mathbb{T}}^M)=\dot{\mathbb{T}}^N.$
\end{enumerate}
    Let us show that $k\colon V\rightarrow N$ lifts in $V[G]$ to a $\lambda$-strongly compact embedding 
    $$k\colon V[G]\rightarrow N[k(G)].$$
    The crux of the matter is to show that $k$ lifts inside $V[G]$.\footnote{There is no hope to lift the embedding $j\colon V\rightarrow M$ inside $V[G]$ in that the iteration $j(\mathbb{P}_\kappa)$ forces (over $V[G]$) $\kappa^{++}$-many Cohen subsets of $\kappa$.} The property of $\kappa$ being $\lambda$-strongly compact is established as follows: Take  $D=i(j``\lambda)$ as a covering set of $k``\lambda$. Note that $D\in N[k(G)]$ and that this set is small (i.e., $|D|^{N[k(G)]}<k(\kappa)$) because $|j``\lambda|^M<j(\kappa)$ and $k(\kappa)=j(\kappa)$.

    \smallskip

Work over $M[G]$. Since $|i(\kappa)|=2^\kappa=\kappa^+$, $(\mathbb{P}^N_{(\kappa,i(\kappa))})_G$ is a $\kappa^+$-closed poset with $\kappa^+$-many antichains. Thus, standard arguments enable to construct a generic filter $G_{(\kappa,i(\kappa))}\in M[G]$ over $N[G]$ for this forcing. Employing the same argument with the Cohen poset $\mathrm{Add}(i(\kappa), 1)_{N[G_{i(\kappa)}]}$ we find a
generic filter $g_{i(\kappa)}\in M[G]$ over $N[G\ast G_{(\kappa,i(\kappa))}]$ for this forcing.

To simplify notations, denote $G\ast G_{(\kappa,i(\kappa))}$ and $G\ast G_{(\kappa,i(\kappa))}\ast g_{i(\kappa)}$   by $G_{i(\kappa)}$ and $G_{i(\kappa)+1}$, respectively. Notice that all generics mentioned so far  belong to the ambient generic extension $V[G].$

\smallskip

It remains to find a generic for $\dot{\mathbb{T}}^N$ inside $V[G].$ For this we use Magidor's \emph{termspace trick}.  Working over $M[G]$, let $\mathbb{A}:=\mathbb{A}(\mathrm{Add}(\kappa,1)_{V[G]}, \dot{\mathbb{T}}^M)$ be the termspace forcing (see Definition~\ref{def: termspace}). In the model $M[G]$, this poset is $\lambda^+$-closed and has the $j(\kappa)$-cc (see Lemma~\ref{lemma: facts termspace}). In particular, since $M[G]$ is closed under $\lambda$-sequences in $V[G]$ and $\mathbb{A}\s M[G]_{j(\kappa)}$, $\mathbb{A}$ is $\lambda^+$-closed in $V[G]$ and has at most $|j(\kappa)|=2^{\lambda}=\lambda^+$-many antichains in $V[G]$. Thus, standard diagonalization arguments enable to construct a generic $H\in V[G]$ for $\mathbb{A}$ over $M[G]$. Since $\mathbb{A}$ is $\lambda^+$-closed in $M[G]$ and $i$ has width ${\leq}\kappa$ we can transfer the generic $H$ to another generic $H^*$ for the termspace forcing
$$\mathbb{A}(\mathrm{Add}(i(\kappa),1)_{N[G_{i(\kappa)}]}, \dot{\mathbb{T}}^N).$$
By virtue of Clause~(2) in Lemma~\ref{lemma: facts termspace} we know that
$$G_{(i(\kappa), j(\kappa))}:=\{\tau_{g_{i(\kappa)}}\mid \tau\in H^*\}$$
is a generic filter for $\dot{\mathbb{T}}^N_{g_{i(\kappa)}}$ over $N[G_{i(\kappa)+1}]$. Moreover, our construction guarantees that the generic $G_{(i(\kappa), j(\kappa))}$ is definable in $V[G]$.

\smallskip

Thereby, we obtain a lifting of $k\colon V\rightarrow N$,
$$k\colon V[G]\rightarrow N[G_{j(\kappa)}]$$
inside the model $V[G].$ Since $\lambda$ was chosen arbitrarily we conclude that $\kappa$ remains strongly compact in $V[G].$

\medskip

(2) Let us show that $\kappa$ is not $2^\kappa$-supercompact in $V[G]$. Suppose it was, and let $j\colon V[G]\rightarrow M[j(G)]$ be a $2^{\kappa}$-supercompact embedding arising from a fine normal measure on $\mathcal{P}^{V[G]}_\kappa(2^\kappa)$. 
Note that $\mathbb{P}_\kappa$ exhibits a \emph{gap} below $\kappa$; indeed
$\mathbb{P}_\kappa\simeq \mathrm{Add}(\gamma,1)\ast \dot{\mathbb{Q}}$
where $\gamma$ is the first measurable cardinal and $\dot{\mathbb{Q}}$ is forced to be (much more than) $\gamma^{+3}$-closed. Thus, by Hamkins' \emph{Gap Forcing Theorem} \cite{HamkinsGapForcing}, $M\s V$ and $j\restriction V=:i\colon V\rightarrow M$ is an elementary embedding definable in $V$. Note that $M^{2^\kappa}\cap V\s M$ and thus $i\colon V\rightarrow M$ witnesses that $\kappa$ is $2^\kappa$-supercompact in $V$. In particular, $M$ thinks that $\kappa$ is measurable and as a result the generic $j(G)$ forces one Cohen subset of $\kappa$ over $V[G]$. This precludes the fact that  $j$ was a member of $V[G].$
\end{proof}

However, after forcing with  $g\s \mathrm{Add}(\kappa,1)_{V[G]}$ the supercompactness of $\kappa$ is restored -- this is a consequence of a classical argument due to Silver (see \cite[Theorem~12.6]{CummingsHandbook} for details).

\smallskip

Let us complete the proof of the theorem by analyzing the $\HOD$ of the generic extensions $V[G]$ and $V[G\ast g]$. By virtue of Lemma~\ref{lemma: self-coding}, $G$ is definable without parameters in $(V[G]_\kappa, V_\kappa, \vec{\mathbb{P}})$ and since $V\s \HOD^{V[G]}$ (recall our assumption of $V\models ``V=\text{gHOD}$"), $$\HOD^{V[G]}=V[G].$$

In addition, we claim that $\HOD^{V[G\ast g]}=\HOD^{V[G]}$. Indeed, 
  $$V\s \HOD^{V[G\ast g]}\s V[G]=\HOD^{V[G]},$$
  where the second inclusion is granted by weak homogeneity of $\add(\kappa,1)_{V[G]}$.  For the other inclusion, $G$ is definable over $(V[G\ast g]_\kappa, V_\kappa,\mathbb{P}_\kappa)$ without parameters because $V[G]_\kappa=V[G\ast g]_\kappa$ (by $\kappa$-distributivity of  the Cohen poset).

\smallskip

To summarize, our final model is $V[G\ast g]$, where $\kappa$ is supercompact. Since $\HOD^{V[G\ast g]}=\HOD^{V[G]}$, $\kappa$ is a strongly compact that is  not $2^\kappa$-supercompact in $\HOD^{V[G\ast g]}$. Finally, the $\HOD$ Hypothesis holds in $V[G\ast g]$ as this is a set-sized  generic extension of a model of $``V=\HOD"$. 

\smallskip

If $\kappa$ was the first supercompact cardinal of $V$ then it remains so in $V[G\ast g]$. This is again a consequence of Hamkins' Gap Forcing Theorem, which implies that forcings admitting a gap below $\kappa$ do not create new supercompact cardinals $\geq \kappa$ (see \cite[Theorem~1.5]{HamkinsDestruction}). 
\end{proof}

\subsection{Transferring large cardinals down to HOD}\label{SectionTransferring}

Woodin showed in \cite{WooPartI} that if $\delta$ is an extendible cardinal and the $\HOD$ hypothesis holds then $\delta$ is supercompact in $\HOD$. The next theorem offers an improvement of this fact by showing that $\delta$ is $C^{(1)}$-supercompact in $\HOD$.

\begin{theorem}\label{HODabsorbsC(1)supercompacts}
Under the $\HOD$ hypothesis, every extendible cardinal is  $C^{(1)}$-supercompact in $\HOD$. 
\end{theorem}
\begin{proof}
Let $\delta$ be an extendible cardinal. The interesting case is when $\delta$ is the first extendible: In any other case, Hamkins Universality Theorem~\ref{UniversalityII} implies that $\delta$ is extendible in $\HOD$, and by \cite{PovOmega} $C^{(1)}$-supercompact.

So going forward we assume that $\delta$ is the first extendible. We would like to employ Poveda's characterization of $C^{(1)}$-supercompactness (Theorem~\ref{CharacterizingCnsupercompacts}) to show that $\delta$ is $C^{(1)}$-supercompact. First, note that that $\delta$ is supercompact in $\HOD$, by Woodin's theorem. (This yields Clause~(1) of Theorem~\ref{CharacterizingCnsupercompacts}.)

\smallskip

To verify Clause~(2) of Theorem~\ref{CharacterizingCnsupercompacts} we argue as follows. 
 Let $j\colon V\rightarrow M$ be an extender embedding with $\crit(j)=\delta$, $M^\delta\s M$, $V_{j(\delta)}\s M$ and $\lambda<j(\delta)$  being inaccessible. (This embedding exists by  Tsaprounis' theorem \cite{Tsan}.) 

   Note that 
   \begin{equation}\label{eqhod}
      \tag{$\ast$}  \HOD^M\cap V_{j(\delta)}=\HOD^{V_{j(\delta)}}\s \HOD.
   \end{equation}
 The first equality follows from $\Sigma_2$-correcteness of $j(\delta)$ in $M$ and  $V_{j(\delta)}\s M$. 
 
 Let $E^*$ be the extender of length $j(\delta)$ inducing $j\colon V\rightarrow M$. Denote by $E$ the $\HOD$-extender of length $j(\delta)$ derived from $j$. Clearly, 
  $\crit(E)=\delta$ and $j_E(\delta)=j(\delta)$ is $\HOD$-inaccessible. So we are left with showing two things:
   \begin{enumerate}[label=(\alph*)]
       \item  $E\in \HOD$;
       \item  $\Ult(\HOD,E)$ is closed under ${<}\delta$-sequences in $\HOD$.
   \end{enumerate}

   \begin{claim}
       Clause~(a) holds.
   \end{claim}
   \begin{proof}
   We will use Woodin's Universality Theorem~\ref{UniversalityI}. This says that  $E\in \HOD$ is equivalent to saying  that for all sets $A\in \HOD\cap \mathcal{P}(j(\delta)^{<\omega})$, $$\text{$j_{E}(A)\cap j(\delta)^{<\omega}\in \HOD.$}$$
    By equation \eqref{eqhod} above, for each $\tau\in \HOD$ with $|\tau|^{\HOD}<j(\delta)$, $$j_{E}(A)\cap j(\delta)^{<\omega}\cap \tau\in \HOD^M\cap V_{j(\delta)}\s \HOD.$$ Since $\HOD$ has the $\delta$-cover and $\delta$-approximation properties (by Theorem~\ref{thm: WoodinCoveringApprx}) it also has the  $j(\delta)$-cover and $j(\delta)$-approximation properties. (For this one uses the fact that $j(\delta)$ is regular.) Therefore, $j_E(A)\cap j(\delta)^{<\omega}\in \HOD$. 
    \end{proof}

    \begin{claim}
        Clause~(b) holds.
    \end{claim}
    \begin{proof}[Proof of claim]
        We use that $\HOD$ has the
    \(\gamma\)-cover property for all sufficiently large strong limit cardinals $\gamma < \delta$. This follows from Goldberg's version of the \(\HOD\) dichotomy theorem for strongly compact cardinals \cite{Gol}. Using this observation,
    the elementarity of $j_{E^*}$, and the fact that $\crit(j_{E^*}) = \delta$, we obtain that
    $j_{E^*}(\HOD)$ has the
    \(\gamma\)-cover property  in \(\Ult(V,E^*)\)   for all sufficiently large strong limit cardinals $\gamma < \delta$. Since \(\Ult(V,E^*)\) is closed under \({<}\delta\)-sequences,
    this means that \(j_{E^*}(\HOD)\) has the \(\delta\)-cover property. 
    
Suppose that $\sigma\in \mathcal{P}_\delta(j_{E^*}(\delta))\cap \HOD$, and we will show that \(\sigma \in j_{E^*}(\HOD)\).
    Let \(\tau\in \mathcal{P}_\delta(j_{E^*}(\delta))\cap j_{E^*}(\HOD)\) cover \(\sigma\). Since \(E\in \HOD\) and 
    $j_{E^*}(\HOD) = 
    \Ult(\HOD, E)$, we have $j_{E^*}(\HOD)\s \HOD$ and therefore \(\tau\in \HOD\). Since
    \(\HOD\cap V_\delta = j_{E^*}(\HOD)\cap V_\delta\), for any set $\pi\in \HOD\cap j_{E^*}(\HOD)$ with $|\pi| < \delta$, \(\mathcal P(\pi)\cap \HOD = \mathcal P(\pi)\cap j_{E^*}(\HOD)\). In particular,
    $$\mathcal{P}(\tau)\cap \HOD = \mathcal{P}(\tau)\cap j_{E^*}(\HOD)$$
    and therefore \(\sigma\in j_{E^*}(\HOD)\), since \(\sigma \in\mathcal{P}(\tau)\cap \HOD\).

   It follows that
    \(\mathcal{P}_\delta(j_E(\delta))\cap \Ult(\HOD,E) = \mathcal{P}_\delta(j_E(\delta))\cap \HOD\),
    and since \(\cf(j_E(\delta)) \geq \delta\),  \(\Ult(\HOD,E)\) is closed under \({<}\delta\)-sequences in \(\HOD\).
    \end{proof}
After this  we have that $j_E\colon \HOD\rightarrow \Ult(\HOD,E)$ witnesses (in $\HOD$) the properties in the characterization of $C^{(1)}$-supercompactness.
\end{proof}
\begin{cor}
    In the $\HOD$ of the model of Theorem~\ref{thm:HOD_thm} the first supercompact is $C^{(1)}$-supercompact.
\end{cor}
\begin{remark}
   The above yields another proof of  the main identity crises  theorem proved by Hayut, Magidor and Poveda in \cite{HMP}. 
\end{remark}

Next, we would like to show that the hierarchies of $C^{(n)}$-supercompact/$C^{(n)}$-extendible cardinals above the first extendible transfer down to $\HOD$. We remind our readers that since we will be working under the assumptions of Theorem~\ref{thm: WoodinCoveringApprx}, $\HOD$ will have the $\delta$-covering and $\delta$-approximation.
\begin{theorem}\label{C(n)areabsorbed}
   Assume $\delta$ is  $\HOD$-supercompact and that the $\HOD$ hypothesis holds. For each $n\geq 1$ and $\lambda\in C^{(n)}$ with $\lambda>\delta$, $\HOD\models ``\lambda\in C^{(n)}$''.
\end{theorem}
\begin{proof}
   The claim is evident for $\Sigma_1$-correct cardinals (i.e., $\beth$-fixed points). So we assume that $n\geq 2$.  Fix  $\lambda>\delta$ a $\Sigma_n$-correct cardinal. By our assumption, $\HOD$ is a weak extender model for the supercompactness of $\delta$ and every such model has the $\delta$-cover and $\delta$-approximation properties.
   
By the Definability Theorem (Theorem~\ref{Definability}), $\HOD$ is $\Sigma_2$ definable (in $V$) taking $P:=\HOD\cap H(\delta^+)$ as parameter. In other words, $\HOD=(W_P)^V$ where $W_P$ is a predicate for the unique inner model with the $\delta$-covering and $\delta$-approximation properties that is $\Sigma_2$ definable via $P$.
   In particular, $$\HOD^{V_\lambda}=\HOD\cap V_\lambda=:\HOD_\lambda.$$

   We will show that $$\HOD_\lambda\prec_{\Sigma_n} \HOD.$$ Let us fix a $\Sigma_n$ formula $\varphi(x)$  and a set $a\in \HOD_\lambda$. 
   \begin{claim}
       $``\HOD\models \varphi(a)$'' is a $\Sigma_n$ formula with parameters $a,P.$
   \end{claim}
   \begin{proof}[Proof of claim]
       The above  is equivalent to saying
   $$\text{$\exists \alpha\in C^{(n-1)}(V_\alpha \models`` W_{P}\models \varphi(a)$'')}.$$
Thus $``\HOD\models \varphi(a)$'' has the claimed complexity. 
   \end{proof}
Suppose first that $\HOD_\lambda\models \varphi(a)$. This amounts to saying 
   $$V_\lambda\models\text{$``\HOD\models \varphi(a)$''}.$$
   By $\Sigma_n$-correctness of $\lambda$ and the previous claim,  $\HOD\models \varphi(a)$.

\smallskip

Conversely, suppose that $\HOD\models \varphi(a)$ holds. 
By the claim this is a true $\Sigma_n(a,P)$ sentence.  By $\Sigma_n$-correctness of $\lambda$ there is  $\alpha\in C^{(n-1)}\cap \lambda$ with
$$V_\alpha\models \HOD\models\varphi(a),$$
with $a,P\in V_\alpha$.
However, $V_\alpha$ and $V_\lambda$ agree on $\Sigma_n$ sentences with parameters in $V_\alpha$, so $\text{$V_\lambda\models ``\HOD\models\varphi(a)$'',}$ which is the same as $\HOD_\lambda\models\varphi(a).$
\end{proof}

\begin{theorem}\label{TrasnferringCnextendibles}
    Assume $\delta$ is $\HOD$-supercompact and that the $\HOD$ hypothesis holds.  If  $\kappa>\delta$ is $C^{(n)}$-extendible (resp. $C^{(n)}$-supercompact) then it is $C^{(n)}$-extendible (resp. $C^{(n)}$-supercompact) in $\HOD.$
\end{theorem}
\begin{proof}
    Fix $\kappa>\delta$ as above. 
    \begin{claim}
        If $\kappa$ is  $C^{(n)}$-supercompact then so it is in $\HOD$.
    \end{claim}
    \begin{proof}[Proof of claim]
We employ the characterization of $C^{(n)}$-supercompactness provided by Theorem~\ref{CharacterizingCnsupercompacts}. {First, $\kappa$ is supercompact in $\HOD$ by Hamkins Univerality Theorem II (Theorem~\ref{UniversalityII}).} Second, for each $\lambda>\delta$ let $j\colon V\rightarrow M$ be an elementary embedding with $\crit(j)=\delta$, $\cf(j(\delta))>\lambda$ and $j(\delta)\in C^{(n)}$. Let $E$ be the $\HOD$-extender of length $j(\delta)$ induced by $j$. By Hamkins Universality Theorem, $j_{E}\colon \HOD\rightarrow N$ is an elementary embedding definable in $\HOD$. By the argument in Theorem~\ref{HODabsorbsC(1)supercompacts}, $N^{<\delta}\cap \HOD\s \HOD$. Finally,  thanks to Theorem~\ref{C(n)areabsorbed}, $\HOD\models j(\delta)\in C^{(n)}$. 
    \end{proof}

    \begin{claim}
         If $\kappa$ is  $C^{(n)}$-extendible then so it is in $\HOD$.
    \end{claim}
    \begin{proof}[Proof of claim]
       We employ Tsaprounis' characterization of $C^{(n)}$-extendibility \cite{Tsan}; namely, $\delta$ is $C^{(n)}$-extendible if and only if for each $\lambda>\delta$ there is an elementary embedding $j\colon V\rightarrow M$ with $\crit(j)=\delta$, $j(\delta)>\lambda$, $M^{\lambda}\s M$, $V_{j(\delta)}\s M$ and $j(\delta)\in C^{(n)}$.  Let $E$ be the $\HOD$-extender of length $j(\delta)$ induced by $j$. By the Universality Theorem II (Theorem~\ref{UniversalityII}), the embedding $j_E\colon \HOD\rightarrow N$ is definable inside $\HOD$ and $\HOD\cap V_{j(\delta)}=N\cap V_{j(\delta)}$. In particular, $\HOD_{j(\delta)}\s N$. Also, {arguing exactly as before, $N^\lambda\cap \HOD\s N$} and $\HOD\models j(\delta)\in C^{(n)}.$ This shows that $\delta$ is $C^{(n)}$-extendible in $\HOD.$
    \end{proof}
    This completes the proof of the theorem.
\end{proof}

\subsection{On cardinal-correct extendible cardinals}\label{sec: CCE}
In \cite{Caicedo} Caicedo introduced \emph{cardinal-preserving embeddings}; namely, elementary embeddings $j\colon M\rightarrow N$ between transitive models such that $\mathrm{Card}^M=\mathrm{Card}^N=\mathrm{Card}$. (Here $\mathrm{Card}$ denotes the class of all cardinals.) 
More recently, Goldberg showed that if there is a proper class of strongly compact cardinals there is no cardinal-preserving elementary embedding $j\colon V\rightarrow M$ \cite{GolComb}. The situation with set-sized transitive models is different. For instance, any elementary embedding between rank initial segments of $V$ is  cardinal-preserving. Inspired by previous work of Magidor and V\"a\"an\"aanen  on L\"owenheim-Skolem-Tarski numbers of strong logics \cite{MagVan}, Galeotti, Khomskii,
and V\"a\"an\"aanen considered \emph{upward} Löwenheim--Skolem--Tarski numbers \cite{gal2020}. To study the upward Löwenheim--Skolem--Tarski number of the
equicardinality logic $\mathcal{L}(I)$, in \cite{GitOsi} the notion of \emph{cardinal-correct extendibility} was proposed:

\begin{definition}[Gitman--Osinski]
    A cardinal $\delta$ is \emph{cardinal-correct extendible} (in short, $\delta$ is $\mathsf{cce}$) if for each $\alpha>\delta$ there is an elementary embedding $j\colon V_\alpha\rightarrow M$ with $\crit(j)=\delta$, $j(\delta)>\alpha$ and $\mathrm{Card}^M=M\cap \mathrm{Card}.$
\end{definition}
 Arguments of Goldberg showed that  $\mathsf{cce}$ cardinals are  strongly compact (see \cite[Proposition~9.5]{GitOsi}). Poveda also  showed that the first $\mathsf{cce}$ cardinal is consistently greater than the first supercompact \cite{PovAxiomA}, and later Gitman and Osinski found an alternate proof \cite{GitOsi}.   Clearly, every extendible cardinal is  $\mathsf{cce}$. The opposite question---namely, whether every $\textsf{cce}$ cardinal must be extendible---is implicit in \cite{GitOsi}. This very question was asked by Osinski to Poveda during a visit to Harvard in the Spring of 2024. As we will show, it turns out that the two notions are not equivalent.  
 The following theorem provides the first key observation:

 \begin{theorem}\label{thm:CCEandHOD}
   Suppose that $\HOD$ is cardinal-correct (i.e., $\HOD^M=\HOD\cap \mathrm{Card}$). If $\delta$ is extendible  
   then it is $\mathsf{cce}$ in $\HOD$.  
 \end{theorem}
 \begin{proof}
       Let $\alpha>\delta$ be a $\Sigma_2$-correct cardinal such that   $\cf(\alpha)>\delta$. Let $j\colon V_{\alpha+1}\rightarrow V_{\beta+1}$ be an elementary embedding witnessing $(\alpha+1)$-extendibility of $\delta$ and $F$ be the \(\text{HOD}\)-extender of length \(\beta\) derived from \(j\). That is,
       $$F=\langle F_a: a\in [\beta]^{<\omega}\rangle$$ where
    $$F_a:=\{X\s [\alpha]^{|a|}: X\in\HOD\,\wedge\, a\in j(X)\}.$$
    Consider the extender ultrapower of \(\HOD\) by $F$ (using only functions in \(\HOD\):  $$\text{$j_{F}\colon \HOD^{ }\rightarrow M_F$.}$$
   
    \begin{claim}\label{FirsCoincidence}
$j_F\restriction \HOD^{ }_{\alpha+1}=j\restriction \HOD_{\alpha+1}$.
    \end{claim}

    \begin{proof}[Proof of claim]
This is a standard argument using that the factor embedding $k\colon (M_F)_{\beta+1}\rightarrow j(\HOD_{\alpha+1})$, defined by 
$$k([a,f]_F):=j(f)(a)$$
has no critical point (since by construction \(\beta\subseteq \text{ran}(k)\)). (To be precise, \(k\) is defined on \([a,f]_F\) whenever \(f : [\alpha]^{|a|}\to V_{\alpha+1}\),
in which case \(f\) is essentially an element of \(V_{\alpha+1}\), so we can make sense of \(j(f)\).)

It follows from this that \(k\restriction (M_F)_\beta\) is onto \(j(\HOD_\alpha)\), from which it immediately follows that \((M_F)_{\beta+1}\subseteq j(\HOD_{\alpha+1})\) and \(k\) is the inclusion map. To see \(k\restriction (M_F)_\beta\) is onto \(j(\HOD_\alpha)\), note that since \(\beta\) is a Beth fixed-point, \((M_F)_\beta\) satisfies that for each \(\xi < \beta\), \((M_F)_\xi\) is the surjective image of an ordinal \(\gamma\). Fixing in \((M_F)_\beta\) a such a surjection \(g: \gamma\to (M_F)_\xi\), we have \[j(\HOD_\alpha)\cap V_\xi = k((M_F)_\xi) = k(g)[k(\gamma)] = k(g)[k[\gamma]] = k[g[\gamma]]\subseteq \text{ran}(k).\qedhere\]
\end{proof}

   Next we argue that $F\in \HOD.$ For this we use {Woodin's Universality Theorem I} (Theorem~\ref{UniversalityI}). We shall also employ Theorem~\ref{thm: WoodinCoveringApprx} saying that if $\delta$ is $\HOD$-supercompact (in our case $\delta$ is extendible) and the $\HOD$ hypothesis holds, then  $\HOD$ has the $\delta$-approximation property. Notice that since $\HOD$ is assumed to be cardinal correct, the $\HOD$ hypothesis must hold, since no successor cardinal is \(\omega\)-strongly measurable in \(\HOD\).

    \begin{claim}
         $F\in \HOD.$
    \end{claim}

    \begin{proof}[Proof of claim]
    By Woodin's Universality Theorem I (Theorem~\ref{UniversalityI}), it suffices to  argue that 
        $$j_F(A)\cap \beta\in \HOD,$$
      for all  $A\in \mathcal{P}(\alpha)\cap \HOD$.
    So fix such an \(A\subseteq \alpha\).

    Since \(A\in \HOD\), \(A\) is amenable to \(\HOD_\alpha\) in the sense that \(A\cap \xi\in \HOD_\alpha\) for all \(\xi < \alpha\). 
    Since \(\alpha\) is \(\Sigma_2\)-correct, we have \(\HOD_\alpha = \HOD^{V_\alpha}\), and so \[j(\HOD_\alpha) = j(\HOD^{V_\alpha}) = \HOD^{V_\beta}\subseteq \HOD\] 
    Moreover by the elementarity of \(j\), \(j(A)\) is amenable to \(j(\HOD_\alpha)\), and so \(j(A)\) is amenable to \(\HOD\).
    
    Since \(\alpha\) has cofinality at least \(\delta\), \(j(\alpha) = \beta\) has cofinality at least \(\delta\). Since \(\HOD\) has the \(\delta\)-approximation property, any \(B\subseteq \beta\)
    that is amenable to \(\HOD\) actually belongs to \(\HOD\). It follows that \(j(A)\in \HOD\), as desired.
    \end{proof}

We deduce from the previous claims that $$j_F\restriction \HOD_\alpha=j\restriction \HOD_\alpha\in \HOD,$$
and recall that, by elementarity and $\Sigma_2$-correctness of $\alpha$, 
$$j\restriction \HOD_\alpha\colon \HOD_\alpha\rightarrow \HOD^{V_\beta}.$$
Thus  $\HOD$ thinks that there is an elementary embedding
$$\iota\colon V_\alpha\rightarrow M:=\HOD^{V_\beta}$$
with $\crit(\iota)=\delta$ and $\iota(\delta)>\alpha$.
\begin{claim}
    $\HOD^{V_\beta}$ is cardinal-correct in $\HOD$.
\end{claim}
\begin{proof}
Since $\HOD^{V_\beta}\s \HOD$, every $\HOD$-cardinal is a $\HOD^{V_\beta}$-cardinal.

  Let us argue the converse. Since $\alpha$ is $\Sigma_2$-correct that yields $$\text{$\HOD^{V_\alpha}= \HOD_\alpha$}.$$ 
  If $\tau$ is a $\HOD_\alpha$-cardinal then it is a $\HOD$-cardinal and thus a $V$-cardinal as well. (Because $\HOD$ is cardinal-correct.) Thus, $\tau$ is a cardinal in $V_\alpha$. 
  
One concludes that
  $$V_\alpha\models \text{``Every $\HOD$-cardinal is a cardinal''}.$$
  By elementarity of $j$, 
    $$V_\beta\models \text{``Every $\HOD$-cardinal is a cardinal''},$$
    hence every $\HOD^{V_\beta}$-cardinal is a cardinal in $V$ and thus a cardinal in $\HOD$ (again, by cardinal-correctness of $\HOD$). 
\end{proof}
The above argument show that in $\HOD$ the cardinal $\delta$ is $\alpha$-$\mathsf{cce}$ for proper-class many $\alpha$'s, thereby we conclude that   $\delta$ is $\mathsf{cce}$ in $\HOD$.
 \end{proof}

 The second key observation to answering Gitman--Osinski's question in the negative is precisely Theorem~\ref{thm:HOD_thm}:

 \begin{theorem}\label{ccenonextendible}
     It is consistent with the $\HOD$ hypothesis that the first strongly compact cardinal is cardinal-correct extendible. In particular, it is consistent for a cardinal-correct extendible not to be extendible.
 \end{theorem}
 \begin{proof}
    Let $V[G]$ be the model of Theorem~\ref{thm:HOD_thm}. Recall that in this model $V\subseteq \HOD^{V[G]}\subseteq V[G]$ and $V$ and $V[G]$ have the same cardinals and cofinalities. This implies that $V[G]$ and $\HOD^{V[G]}$ have the same cardinals and cofinalities. Moreover $\delta$ is extendible in $V[G]$, so by Theorem~\ref{thm:CCEandHOD}, $\delta$ is $\mathsf{cce}$ in $\HOD^{V[G]}$. But by the proof of Theorem \ref{thm:HOD_thm}, $\delta$ is the least strongly compact cardinal of \(\HOD^{V[G]}\). 
 \end{proof}
 The crux of the matter of the previous argument is Theorem~\ref{thm:CCEandHOD}. This was proved under the assumption that $\HOD$ is cardinal-correct, which is stronger than asserting the $\HOD$ hypothesis. Therefore, we ask:
\begin{question}
    Assume the $\HOD$ hypothesis. If $\delta$ is $\textsf{cce}$ must it be so in $\HOD$?
\end{question}
Goldberg and Poveda conjecture that the answer is negative. 

\smallskip

Another question highlighted by our analysis is the following:
\begin{question}[Gitman--Osinski]
  Is it consistent for the first $\textsf{cce}$ to be the first measurable cardinal?
\end{question}

There is a natural strengthening of cardinal-correct extendibility for which the above question admits a negative answer:

\begin{definition}
    A cardinal $\delta$ is \emph{stationary-correct extendible} (in short, $\delta$ is $\mathsf{sce}$) if for each $\alpha>\delta$ there is an elementary embedding $j\colon V_\alpha\rightarrow M$ with $\crit(j)=\delta$, $j(\delta)>\alpha$ and $M$ is stationary-correct.\footnote{I.e., every $M$-stationary is stationary.}
\end{definition}
Clearly, every extendible is $\mathsf{sce}$ and the next observation shows that the first $\mathsf{sce}$ cannot be the first measurable.
\begin{prop}
   If $\delta$ is $\mathsf{sce}$ then it is supercompact. 
\end{prop} 
\begin{proof}
Fix $\lambda > \delta$ a regular cardinal. Let $\kappa > \lambda$ regular and $\langle S_\alpha \mid \alpha < \lambda \rangle$ be a partition of $E^\kappa_\omega$ into stationary sets. Since $\delta$ is stationary-correct extendible there is a large enough $\theta$ and an embedding $j\colon V_\theta \to M$ with $\mathrm{crit}(j) = \delta$, $M$ stationary-correct and $\langle S_\alpha \mid \alpha < \lambda \rangle \in V_\theta$. Set $$\langle T_\alpha \mid \alpha < j(\lambda) \rangle := j(\langle S_\alpha \mid \alpha < \lambda \rangle).$$

Let $\kappa_* := \sup(j``\kappa)$. Since $j(\kappa)$ is regular in $M$ so is $\kappa_*$ in $V$ as well. Thus, $\kappa$ is a discontinuity point of $j$. In addition, $\mathrm{cf}^M(\kappa_*) = \mathrm{cf}(\kappa_*) > \omega$.

\medskip

By a classical argument of Solovay, 
\[
j``\lambda = \{ \alpha < j(\lambda) \mid M \models T_\alpha \cap \kappa_* \text{ is stationary} \}.
\]
which yields $j``\lambda \in M$, and thus
\[
\mathcal{U} = \{ X \in \mathcal{P}_\delta(\lambda) \mid j``\lambda \in j(X) \}
\]
is a $\lambda$-supercompact measure on $\delta.$
\end{proof}
The argument of Theorem~\ref{thm:CCEandHOD} applies verbatim if one replaces $``\HOD$ is cardinal-correct" by $``\HOD$ is stationary-correct" – in this latter case one obtains the stronger result that the first extendible is $\mathsf{sce}$ in $\HOD.$

\smallskip

Similarly, the result in  Theorem~\ref{thm:HOD_thm} can be obtained while ensuring that $\HOD^{V[G]}$ is stationary correct in $V[G]$. The idea is to replace the two-step iteration $\mathbb{NR}\ast \dot{\mathbb{CS}}$ with the two-step iteration which first adds a Suslin tree and then forces with the Suslin tree itself (see e.g., \cite{CFH}). Thus, in this case we obtain a model where $\delta$ is extendible, $\delta$ is the first strongly compact cardinal in $\HOD$ and $\HOD$ is stationary-correct. Using the variation of Theorem~\ref{thm:CCEandHOD} mentioned above we conclude the following:
\begin{theorem}
     It is consistent with the $\HOD$ hypothesis that the first strongly compact cardinal is the first stationary-correct extendible. \qed
\end{theorem}

\section{On the HOD Dichotomy for supercompact cardinals}\label{sec: on the HOD dichotomy for supercompact}
In this section we analyze the optimality of the version of the HOD Dichotomy theorem proved by the first author in \cite{Gol}.  First, in \S\ref{sec: where weak covering takes a hold} we analyze the optimality of the hypothesis of Goldberg's theorem.  Specifically, we construct a model where
$$\text{$``\delta$ is the first supercompact + $\HOD$ hypothesis''}$$
and there is  a club of cardinals $\kappa<\delta$ such that $$\text{$\kappa$ is regular in $\HOD$ and }\kappa^{+\HOD}<\kappa^+.$$
This shows that the $\HOD$ dichotomy theorem cannot be proved from any large cardinal hypothesis rank reflected by a supercompact cardinal. For instance, it does not follow from a proper class of strong cardinals, etc. In addition, we answer a question due to Cummings, Friedman and Golshani  \cite{CumFriGol} relative to failure of weak covering and supercompact cardinals.

\smallskip

Next, in \S\ref{sec: where omega strongly} we examine the optimality of the conclusion of Goldberg's version of the $\HOD$ dichotomy theorem. Goldberg's result implies that if $\delta$ is supercompact and the $\HOD$ hypothesis fails, then all sufficiently large regular cardinals are $\omega$-strongly measurable in $\HOD$. We show that this cannot be improved to the conclusion that every regular cardinal greater than or equal to $\delta$ is $\omega$-strongly measurable in $\HOD$. Moreover, we show that for each cardinal $\lambda\geq \delta$, there is a set-forcing extension in which $\delta$ remains supercompact and the power-set of $\lambda$ is contained in $\HOD$. In particular, since the failure of the $\HOD$ hypothesis is preserved by set-sized forcing, we obtain a model where the $\HOD$ hypothesis fails, $\delta$ is supercompact and $\eta_0>\lambda$ where $\eta_0$ is defined as follows:
    $$\eta_0:=\min\{\theta\geq \delta\mid \text{$\theta$ is $\omega$-strongly measurable in $\HOD$}\}.$$
    
    Note that by the $\HOD$ dichotomy proved by Goldberg in \cite{Gol}, every regular cardinal above $\eta_0$ must be $\omega$-strongly measurable in $\HOD$. 

\subsection{Where weak covering takes hold}\label{sec: where weak covering takes a hold} We begin this section 
addressing the following question from Cummings et al. {\cite[p.31]{CumFriGol}}:

\begin{question*}[Cummings--Friedman--Golshani]
Is it consistent for $\delta$ to be supercompact and $\kappa^{+\HOD}<\kappa^+$ for all $\kappa<\delta$?    
\end{question*}

We argue that, at least under the $\HOD$ hypothesis,  there must be \emph{many} agreements between $V$ and $\HOD$ on singular cardinals and their successors. (In \cite{CumFriGol}, the authors mention that Woodin conjectured a negative answer to their question, granting the $\HOD$ conjecture. Thus, it seems likely to us that Woodin was aware of the result we are about to prove.)

We shall employ the following lemma, which is a slight fix of \cite[Theorem~4.1]{Cheng}.\footnote{In \cite[Theorem~4.1]{Cheng}  it is claimed that $j_2(N\cap V_{\lambda_2})=N\cap V_\lambda,$ but this does not hold. For instance, if $N:=\HOD\cap V_\delta$ then $\delta_2\in N\cap V_{\lambda_2}$ and $j_2(\delta_2)=\delta$ yet $\delta\notin N\cap V_{\lambda}=N.$}

\begin{lemma}\label{lemma: Magidor supercompactness}
Let $N\s V_\delta$. The following assertions are equivalent:
\begin{enumerate}
    \item $\delta$ is supercompact.
    \item For each $\lambda>\delta$  and $a\in V_\lambda$ there are elementary embeddings
    $$j_1\colon V_{\lambda_1+1}\rightarrow V_{\lambda_2+1}\;\text{and}\; j_2\colon V_{\lambda_2+1}\rightarrow V_{\lambda+1}$$
    and sets $a_1$ and $a_2$ such that:
\begin{enumerate}
    \item $\crit(j_1)=\delta_1$, $\crit(j_2)=\delta_2$, $j_1(\delta_1)=\delta_2$, $j_2(\delta_2)=\delta$.
    \item $a_1\in V_{\delta_1}$, $a_2\in V_{\delta_2}$, $j_1(a_1)=a_2$ and $j_2(a_2)=a$.
    \item $j_1(N\cap V_{\lambda_1})=N\cap V_{\lambda_2}$.\qed
\end{enumerate}
\end{enumerate}
\end{lemma}
\begin{remark}
   For each triple $\langle a,\lambda, \delta_1\rangle$, let us denote by $\varphi(a,\lambda,\delta_1)$ the formula asserting the existence of embeddings $j_1$ and $j_2$, cardinals $\lambda_1$, $\lambda_2$, $\delta_2$ and sets $a_1$ and $a_2$ as in Clause~(2) above. One can show that $\{\delta_1<\delta: \varphi(a,\lambda,\delta_1)\}$ is stationary in $\delta$. This is because this  belongs to the normal measure on $\delta$ induced by the second ultrapower of a normal fine measure on $\mathcal{P}_\delta(|V_{\lambda+1}|).$
\end{remark}

It is well-known  that if $N$ is a weak extender model for $\delta$  supercompact then every singular cardinal $\gamma>\delta$ is singular in $N$ and $(\gamma^+)^{\HOD}=\gamma^+$ (\cite[\S2.2]{midrasha}). A close inspection of  Lemmas~2.8, 2.9 and 2.10 in \cite{midrasha} makes clear that this can be obtained level-by-level. Specifically:
\begin{lemma}\label{lemma: covering}
    Let $\delta<\sigma$ be cardinals with $\cf(\sigma)<\sigma$. If $N$ is a weak extender model for $\delta$ is $\sigma$-supercompact then $\sigma$ is singular in $N$ and $(\sigma^+)^N=\sigma^+$.\qed
\end{lemma}


\begin{theorem}\label{thm: many instances of covering}Assume the $\HOD$ hypothesis holds. If $\delta$ is a supercompact cardinal then there is a stationary set $S\s \delta$ such that the $\omega$th-successor of each point in $S$ is singular in $\HOD$ and $\HOD$ is correct about its successor.
\end{theorem}
\begin{proof}
Let $\delta<\mu<\kappa$ be cardinals that are (respectively) a $\beth$-fixed point and a regular non-$\omega$-strongly measurable in HOD above $2^\mu$. (This choice is possible by virtue of the HOD hypothesis.) Thus, there is a stationary partition  $\langle S_\alpha: \alpha<\mu\rangle\in \HOD$  of $E^\kappa_\omega$.
Set $a:=\{\kappa,\mu, \langle S_\alpha: \alpha<\mu\rangle\}$ and let $\lambda>\mu$ be a $\Sigma_2$-correct cardinal.  Applying Lemma~\ref{lemma: Magidor supercompactness} (and the remark following it) to $\langle \lambda,a\rangle$ we get a stationary set $S\s \delta$ (which depends upon $\langle \lambda, a\rangle$) such that for each $\delta_1\in S$ there are  elementary embeddings
$$j_1\colon V_{\lambda_1+1}\rightarrow V_{\lambda_2+1}\;\text{and}\; j_2\colon V_{\lambda_2+1}\rightarrow V_{\lambda+1}$$
with $\crit(j_1)=\delta_1$, $\crit(j_2)=\delta_2$, $j_1(\delta_1)=\delta_2$ and $j_2(\delta_2)=\delta$ and
$$j_1(\HOD\cap V_{\lambda_1})=\HOD\cap V_{\lambda_2}.$$

These embeddings also do come together with respective \emph{reflections} of the members of $a$; namely, $\langle \kappa_1, \mu_1, \langle S^1_\alpha: \alpha<\mu_1\rangle\rangle$ and $\langle \kappa_2, \mu_2, \langle S^2_\alpha: \alpha<\mu_2\rangle\rangle$.

\smallskip

Since $\langle S_\alpha: \alpha<\mu\rangle\in \HOD\cap V_\lambda=\HOD^{V_\lambda}$ (by $\Sigma_2$-correctness of $\lambda$) elementarity of $j_2$ implies that $\langle S^2_\alpha: \alpha<\mu_2\rangle\in \HOD^{V_{\lambda_2}}\s \HOD$.

Next look at the normal measure on $\mathcal{P}_{\delta_1}(\delta_1^{+\omega})$ induced by $j_1$; namely,
$$\mathcal{U}:=\{X\s \mathcal{P}_{\delta_1}(\delta_1^{+\omega}): j_1``\delta_1^{+\omega}\in j_1(X)\}.$$

\begin{claim}
    $\mathcal{U}\cap \HOD\in \HOD$ and $\mathcal{P}_{\delta_1}(\delta_1^{+\omega})\cap \HOD\in \mathcal{U}$. 
\end{claim}
\begin{proof}[Proof of claim]
Clearly, $\mathcal{U}\cap \HOD$ is definable via $j_1``\delta_1^{+\omega}$ and $j_1\restriction \HOD\cap V_{\mu_1}$, so $\mathcal{U}\cap \HOD\in \HOD$ as long as the aforementioned objects belong to $\HOD$.

On the one hand,
$$j_1``\mu_1=\{\alpha<\mu_2: S^2_\alpha\cap (\sup j_1``\kappa_1)\; \text{is stationary}\}\in \HOD$$
in that  $\langle S^2_\alpha:\alpha<\mu_2\rangle\in \HOD.$ In particular, $j_1``\delta_1^{+\omega}\in \HOD\cap V_{\lambda_2}$.

On the other hand, fix $f\colon \mu_1\rightarrow \HOD\cap V_{\mu_1}$  a bijection in $\HOD$ (this is possible because $\mu_1=\beth_{\mu_1}$). In fact, $f\in \HOD\cap V_{\lambda_1}$ (as $\lambda_1$ is strong limit) and thus $j_1(f)\in \HOD\cap V_{\lambda_2}\s \HOD$.  Since $j_1\restriction \HOD\cap V_{\mu_1}$ is definable via $j_1``\mu_1$ and $j_1(f)$ one concludes that it belongs to $\HOD$.

The above shows that $\mathcal{U}\cap \HOD\in \HOD$. Finally note that $\mathcal{P}_{\delta_1}(\delta_1^{+\omega})\cap \HOD\in \mathcal{U}$ because this is equivalent to saying  $\mathcal{P}_{\delta_1}(\delta_1^{+\omega})\cap \HOD\cap V_{\lambda_1}\in \mathcal{U}$ and this latter assertion holds in that
$j_1``\delta_1^{+\omega}\in j_1(\HOD\cap V_{\lambda_1})=\HOD\cap V_{\lambda_2}.$
\end{proof}
The above claim shows that $\HOD$ is a weak extender model for $``\delta_1$ is $\delta_1^{+\omega}$-supercompact'' and therefore $\delta_1^{+\omega}$ is singular in $\HOD$ and its successor is correctly computed (by Lemma~\ref{lemma: covering}). Since we have argued this for each $\delta_1\in S$ the theorem follows.
\end{proof}
Thereby we answer the question in \cite{CumFriGol} in the negative:
\begin{cor}
    If $\delta$ is supercompact and the $\HOD$ hypothesis holds then for unboundedly many cardinals $\kappa$, $\cf^{\HOD}(\kappa)=\cf(\kappa)$ and $\kappa^{+\HOD}=\kappa^+$.
\end{cor}

Next we would like to show that Theorem~\ref{thm: many instances of covering} is arguably optimal:

\begin{theorem}\label{thm: extendingCumFriGol}
Suppose that there is a supercompact cardinal $\delta$ and that the $\mathrm{GCH}$ holds. Then, there is a generic extension where the following  hold:
    \begin{enumerate}
        \item The $\HOD$ hypothesis.
        \item $\delta$ is the first supercompact.
        \item There is a club $D\s \delta$ consisting of cardinals $\kappa$ such that:
        \begin{enumerate}
            \item $\kappa^{+\HOD}<\kappa^+$;
            \item $\HOD\models``\kappa$ is regular".
        \end{enumerate}
    \end{enumerate}
\end{theorem}
In \cite{CumFriGol}, Cummings, Friedman and Golshani construct a model where the $\HOD$ hypothesis holds and a  supercompact $\delta$ carries a club of cardinals $\kappa$ such that $\kappa^{+\HOD}<\kappa^+$. Building upon this work we will be able to prove the above-mentioned theorem.  
The model under consideration is a generic extension via \emph{Supercompact Radin forcing}, $\mathbb{R}^{\sup}_u$. The definition of $\mathbb{R}^{\sup}_u$ is rather technical so we refrain to reproduce  it here. Instead, we refer the reader to \cite{CumFriGol} (indeed, we recommend our readers to have a copy of \cite{CumFriGol} close to them) for a complete account of the poset.

\smallskip

In \cite{CumFriGol} a pair $(u,A)$ is called \emph{good} if $u\in\mathcal{U}^{\sup}_\infty$, $A\s \mathcal{U}^{\sup}_\infty\cap S(\kappa_u,\lambda_u)$ and $A\in \mathcal{F}(u)$. Following \cite[\S4]{CumFriGol}, given a good pair $(u,A)$ we denote 
\begin{itemize}
    \item $\phi(u):=\langle \kappa_u\rangle^\smallfrown \langle u(\alpha): 1\leq \alpha<\ell(u)\rangle$,
    \item $\phi(A):=\{\phi(v): v\in A\}.$
\end{itemize}
(Note that formally speaking $\phi$ is possibly a class function.)
Following \cite[\S4]{CumFriGol}, denote $\mathcal{U}^{\mathrm{proj}}_\infty=\{\phi(u): u\in \mathcal{U}^{\sup}_\infty\}$. Similarly, write $$\mathcal{U}^{\mathrm{proj},-}_\infty:=\{\phi(u)\restriction\ell_u: u\in \mathcal{U}^{\sup}_\infty\}.$$
In \cite[\S4]{CumFriGol} the authors call a pair $(w,B)$ \emph{good for projected forcing} if $$\text{$w\in \mathcal{U}^{\mathrm{proj}}_\infty$, $B\s \mathcal{U}^{\mathrm{proj}}_\infty\cap V_{\kappa_w}$ and $\textstyle B\in \bigcap_{0<\alpha<\ell(w)} \phi(w(\alpha))$.}$$

Let $\mathbb{R}^{\text{proj}}_{\phi(u)}$ 
be defined as the usual Supercompact Rading forcing yet employing good pairs for  projected forcing \cite[Definition~4.4]{CumFriGol}. 
Intuitively speaking, $\mathbb{R}^{\text{proj}}_{\phi(u)}$ is like regular Radin forcing except for a difference---the measures $w(\alpha)$ on the measure sequences $w$ appearing in a condition $p\in \mathbb{R}^{\text{proj}}_{\phi(u)}$ are on $S(\kappa_w,\lambda_w)$ and not on $V_{\kappa_w}$. As it turns out that, if $\langle \langle \sigma_\alpha\rangle^\smallfrown w_\alpha\mid \alpha<\theta\rangle$ is the generic object introduced after forcing with  $\mathbb{R}^{\sup}_u$ then $\langle w_\alpha\mid \alpha<\theta\rangle$ is  the generic object introduced after forcing with the projected forcing  $\mathbb{R}^{\text{proj}}_{\phi(u)}$. 
In \cite[\S4]{CumFriGol}, the authors show that $\phi$ induces a (weak) projection between the supercompact Radin forcing $\mathbb{R}^{\sup}_u$ and its projected forcing $\mathbb{R}^{\text{proj}}_{\phi(u)}$. Thus, every $V$-generic $G$ for $\mathbb{R}^{\sup}_u$ induces a generic filter $G^\phi$ for $\mathbb{R}^{\text{proj}}_{\phi(u)}$. Also, they show that  $\mathbb{R}^{\text{proj}}_{\phi(u)}$ preserves cardinals. 

\smallskip

The following is a well-known result (see \cite[Theorem~5.15]{Gitik-handbook}): 
\begin{lemma}\label{lemma: preserving supercompactness}
    Let $w\in \mathcal{U}^{\sup}_\infty$ be such that $\langle \kappa_w\rangle^\smallfrown \langle \phi(w(\alpha)): 1\leq \alpha<\ell(w)\rangle$ has a repeat point.\footnote{Recall that a \emph{repeat point} point for $\langle \kappa_w\rangle^\smallfrown \langle \phi(w(\alpha)): 1\leq \alpha<\ell(w)\rangle$ is an ordinal $\gamma<\ell(w)$ for which $\bigcap_{1\leq \beta<\gamma}\phi(u(\beta))=\bigcap_{1\leq \beta<\ell(u)}\phi(u(\beta))$.} Then, $\mathbb{R}^{\mathrm{proj}}_{\phi(w)}$ forces $``\kappa_w$ is measurable".\qed
\end{lemma}

After the above explanations we are in a position to prove Theorem~\ref{thm: extendingCumFriGol}. Suppose the following properties hold in a ground model $V$:
\begin{enumerate}
    \item GCH.
    \item $\delta$ is supercompact.
    \item $V=g\HOD$.
\end{enumerate}

    The above configuration can be obtained from GCH and the existence of a supercompact cardinal $\delta$: First, make $\delta$ indestructible under $\delta$-directed-closed forcings that preserve the GCH. (This forcing itself preserves the GCH by Lemma~\ref{GCHplussigma2correctness}) Then force with the $\diamondsuit^*$ Oracle partial forcing of \cite[\S3]{Broo2} above the supercompact $\delta$ to get a model where GCH and $V=g\HOD$ hold.

\smallskip

   Let $u\in \mathcal{U}^{\sup}_\infty$ be a $(\delta,\delta^+)$-supercompact measure sequence with length $\delta^{++}<\ell(u)<\delta^{+3}$ such that forcing with $\mathbb{R}^{\sup}_u$ preserves supercompactness of $\delta$ (see \cite[Remark~5.1]{CumFriGol}). Fix $G\s \mathbb{R}^{\sup}_u$ a $V$-generic filter and denote
$$\mathcal{S}_G:=\{w\in \mathcal{U}^{\sup}_\infty: \exists p\in G\; (\text{$w$ appears in $p$)}\}$$
and
$$C=\langle \kappa_\alpha\mid \alpha<\delta\rangle$$
the increasing enumeration of  $\{\kappa_w\mid \exists p\in G\,\text{($w$ appears in $p$ and $w\neq u$)}\}$.

\begin{remark}
   Since $\lambda_u=\delta^+$   by forcing below an appropriate condition we may assume without loss of generality that $\lambda_w=\kappa^+_w$ for all supercompact measure sequence $w$ appearing in a typical condition in $\mathbb{R}^{\sup}_u$.
\end{remark}

\smallskip

We claim that $V[G]$ is the model witnessing Theorem~\ref{thm: extendingCumFriGol}:
\begin{proof}[Proof of Theorem~\ref{thm: extendingCumFriGol}]
 Clause~(1) is obvious. To establish Clause~(2) (i.e., that $\delta$ is the first supercompact cardinal in $V[G]$) we use a theorem of Magidor and Sinapova (see Theorem~3.1 in \cite{MagSin}). The argument goes as follows:  Let $\kappa\in C$ be a cardinal with $V[G]$-cofinality $\omega$. Since $$\text{$\cf^{V[G]}(\kappa^{+V})=\omega$ and $(\kappa^{++V})=(\kappa^{+V[G]})$}$$
 (by the properties of $\mathbb{R}^{\sup}_u$) the assumptions of \cite[Theorem~3.1]{MagSin} are met. Thus $\square_{\kappa,\omega}$ holds in $V[G]$ for all such $\kappa$. Since this comment applies for unboundedly-many cardinals $\kappa<\delta$, there are no  strongly compact cardinals below $\delta.$ (This is a well-known fact about square-like sequences and strongly compact cardinals, c.f. \cite{CumForMag}.)

 \smallskip

 The basic idea behind (3) is that the supercompact Radin forcing $\mathbb{R}^{\sup}_u$ introduces many supercompact sequences $w$ whose associated measure sequence $\langle \kappa_w\rangle^\smallfrown \langle \phi(w(\alpha)): 1\leq \alpha<\ell(w)\rangle$ on $V_{\kappa_w}$ contains a repeat point. Employing Lemma~\ref{lemma: preserving supercompactness}, we will infer that forcing with $\mathbb{R}^{\text{proj}}_{\phi(w)}$ preserves 
 the regularity of $\kappa_w$. Finally, using the factoring lemma of the projected forcing, we will  conclude that the same applies to the generic extension by $\mathbb{R}^{\text{proj}}_{\phi(u)}$.
\begin{claim}
    The set $D:=\{\kappa_w: w\in \mathcal{S}_G\,\wedge\, \ell(w)\geq \kappa^{++}_w\}$ is a club in $\delta$.
\end{claim}
\begin{proof}
Unboundedness follows by density. The key observation is that for each $\alpha<\delta$,
$\mathcal{D}_\alpha:=\{p\in \mathbb{R}^{\sup}_u: \text{$p$ mentions $w$ with $\ell(w)\geq \kappa_w^{++}$ and $\kappa_w> \alpha$}\}$. (Recall that our supercompact measure sequence $u$ has $\ell(u)> \delta^{++}$.)

For closure we argue as usual. Let $\dot{D}$ be a name for the above-displayed set and suppose that $\alpha<\delta$ and $p\in \mathbb{R}^{\sup}_u$ forces $``\alpha\notin \dot{D}"$. 
By definition, there is no $w$ appearing in $p$ such that both $\ell(w)\geq \kappa_w^{++}$ and $\kappa_w=\alpha$.

\smallskip

\underline{Case 1:} Suppose that there is $w$ appearing in $p$ with $\kappa_w=\alpha$. Say,
$$p=\langle (u_0,A_0),\dots, (w,B), \dots, (u,A)\rangle.$$

Then it must be that $\ell(w)<\kappa_w^{++}$.  Let $X$ be the set of supercompact measure sequences $v\in \mathcal{U}^{\sup}_\infty$ for which there is $C$ such that $(v,C)$ is a good pair and the sequence
$$\langle (u_0,A_0),\dots,(v,C), (w,B), \dots, (u,A)\rangle$$
is an extension of $p$. It turns out that $\pi_{w}[X]\in \mathcal{F}(w)$ (\cite[Lemma~3.7]{CumFriGol}), where $\pi_w$ is the unique order isomorphism between $w(0)$ and $\lambda_w:=\otp(w(0))$. In particular, we can let $p^*\leq^* p$ be the condition resulting from replacing in $p$ the measure one set  $B$ by $$B^*:=B\cap \pi_{w}[X]\cap \{v\in S(\kappa_w,\kappa_w^+): \ell(v)<\kappa_v^{++}\}.$$ This makes sure that $p^*$ forces $``\sup(\dot{D}\cap \alpha)<\alpha"$, as needed.

\smallskip

\underline{Case 2:} Suppose that there is no $w$ appearing in $p$ with $\kappa_w=\alpha$. Let $w$ be the first measure sequence in $p$ with $\kappa_w>\alpha$. 
As in the previous argument, set $B^*:=B\cap \pi_{w}[X]\cap \{v\in S(\kappa_w,\kappa_w^+): \kappa_v>\alpha\}$ and let $p^*$ be defined as before. The resulting condition $p^*\leq^* p$ forces $``\sup(\dot{D}\cap \alpha)<\alpha".$
\end{proof}
Note that the $V$-successor cardinal of each $\theta\in D$ is collapsed in $V[G]$: Let $\theta\in D$ and $w\in \mathcal{S}_G$ be the unique sequence such that $\kappa_w=\theta$. Let $p\in G$ be a condition in the generic of the form $\langle (w,B),(u,A)\rangle$. By the factoring property of supercompact Radin forcing,
$$\mathbb{R}^{\sup}_u/p\simeq \mathbb{R}^{\sup}_w/\langle (w,B)\rangle\times \mathbb{R}^{\sup}_u/\langle (u,A^*)\rangle$$
where $A^*:=\{v\in A: w\prec v\}$. By standard properties of the Supercompact Radin forcing, the first of these factors forces $``\theta$ is a cardinal", $``|\theta^{+V}|=\theta"$ and $``\theta^{++V}$ remains a cardinal". Note that the second factor is distributive enough so as to not alter these properties. By the analysis of \cite[\S4]{CumFriGol},
$$V\s \HOD^{V[G]}\s V[G^\phi],$$
and $V[G^\phi]$ is a cardinal-preserving generic extension of $V$.

Therefore, $$V[G]\models ``\theta^{+\HOD^{V[G]}}=\theta^{+V}<\theta^{+}".$$ 

The above argument disposes with the first bullet of (3) of the theorem. As for the $\theta$ being regular in $\HOD$ we argue as follows:  Since the GCH holds below $\theta$ and $\ell(w)\geq \theta^{++}$ the measure sequence $$\langle\theta\rangle^\smallfrown\langle \phi(w(\alpha)): 1\leq \alpha<\ell(w)\rangle$$ on $V_{\theta}$ has a repeat point. Thus, by Lemma~\ref{lemma: preserving supercompactness} above, $\theta$ remains measurable in the generic extension by $\mathbb{R}^{\text{proj}}_w$. Forcing with the projected forcing below a condition of the form $\langle (\phi(w),B),(\phi(u),A)\rangle$ we know that
$$\mathbb{R}^{\text{proj}}_{\phi(u)}/\langle (\phi(w),B),(\phi(u),A)\rangle\simeq \mathbb{R}^{\text{proj}}_{\phi(w)}/\langle (\phi(w),B)\rangle\times \mathbb{R}^{\text{proj}}_{\phi(u)}/\langle (\phi(u),A^*)\rangle$$
(see \cite[Lemma~3.6]{CumFriGol}).
Note that the second factor is distributive enough to preserve regularity of $\theta$. Therefore we deduce that $\theta$ is regular. In particular, $\theta$ is regular in $V[G^\phi]$, hence in $\HOD^{V[G]}$ as well.
\end{proof}

\begin{question}\label{que: omegastronglymeasurable}
    Assume the $\HOD$ hypothesis holds and that $\delta$ is the first supercompact cardinal. Can $\delta$ be $\omega$-strongly measurable in $\HOD$?
\end{question}
By \cite[Theorem~2.11]{Gol} if $\delta$ is regular and $N$ is an inner model such that $\mathrm{Reg}^N\cap \delta$ contains an $\omega$-club then it actually contains a club. Therefore, a positive answer to Question~\ref{que: omegastronglymeasurable} yields a stronger configuration than the one arranged in Theorem~\ref{thm: extendingCumFriGol}.  In contrast, by \cite[Propositive~2.3]{Gol}, if the $\HOD$ hypothesis holds and $\delta$ is strongly compact no regular cardinal strictly above $\delta$ is $\omega$-strongly measurable in $\HOD$.

\smallskip

A related open question is whether Theorem~\ref{thm: extendingCumFriGol} can be established below the first $\HOD$-supercompact cardinal $\delta$. This is particularly interesting because, under the $\HOD$ hypothesis and the existence of a $\HOD$-supercompact, $\HOD$ exhibits a behavior similar to that under the existence of an extendible cardinal (e.g., $\HOD$ satisfies $\delta$-covering and $\delta$-approximation). Nonetheless, it is known that a configuration like that of Theorem~\ref{thm: extendingCumFriGol} is impossible below the first extendible cardinal. Therefore, it is reasonable to ask the following:

\begin{question}\label{que: HOD supercompactness and covering}
 Is it possible to get a model where the $\HOD$ hypothesis holds, $\delta$ is  $\HOD$-supercompact and there is a club $C\s\delta$ consisting of regular cardinals in $\HOD$?
\end{question}
Also, it is not difficult to produce a model in which $\delta$ is $C^{(n)}$-supercompact, the $\HOD$ hypothesis holds, and there are unboundedly many $\kappa < \delta$ such that $\kappa^{+\HOD} < \kappa^+$. The idea is to force with a Magidor product of supercompact Prikry forcings on a sufficiently sparse set of cardinals $\kappa < \delta$ that are $\kappa^+$-supercompact. The resulting forcing is cone-homogeneous and preserves $C^{(n)}$-supercompacts by the arguments in \cite{HMP}. The situation becomes less clear if one attempts to arrange this configuration on a club of cardinals $\kappa$. In other words, it remains open  whether the Cummings--Friedman--Golshani configuration \cite{CumFriGol} can be extended to the case where $\delta$ is $C^{(n)}$-supercompact. In fact, more can be said: in the model of \cite{CumFriGol}, $\delta$ is not $C^{(1)}$-supercompact cardinal per the argument in \cite[Theorem~5.2]{PovAxiomA}.  

\begin{question}\label{que: C(1)supercompacts and covering}
    Is it consistent for the first $C^{(1)}$-supercompact $\delta$ to carry a club $C\s \delta$ where $\kappa^{+\HOD}<\kappa^+$? 
\end{question}

\subsection{Where $\omega$-strong measurability takes hold}\label{sec: where omega strongly}

This section pertains to Goldberg's result
\cite{Gol} that if the HOD hypothesis fails and there is a strongly compact cardinal, then all sufficiently large regular cardinals are \(\omega\)-strongly measurable in \(\HOD\).

\smallskip

For the purposes of this subsection, let us assume that \(\kappa\) is the least strongly compact cardinal and the HOD hypothesis fails. Let \(\eta_0\) be the least ordinal greater than or equal to \(\kappa\) such that all regular cardinals greater than or equal to \(\eta_0\) are \(\omega\)-strongly measurable in \(\HOD\). What can be said about the size of \(\eta_0\)? Recall that Woodin's HOD dichotomy shows that if \(\delta\) is extendible, then all regular cardinals greater than or equal to \(\delta\) are \(\omega\)-strongly measurable in \(\HOD\) \cite{midrasha}. Therefore \(\eta_0\leq \delta\). In fact, definability considerations imply that \(\eta_0 < \delta\). This is because \(\eta_0\) is \(\Sigma_3\)-definable from the parameter \(\kappa\), and so in fact \(\eta_0\) is strictly less than the least \(\Sigma_3\)-reflecting cardinal above \(\kappa\), which is well below the first extendible cardinal \(\delta\).

The analysis of \cite{Gol}
shows that \(\HOD\) is quite close to \(V\)
between \(\kappa\) and \(\eta_0\). For example, no regular cardinal between \(\kappa\) and \(\eta_0\) is \(\omega\)-strongly measurable in \(\HOD\). 
Moreover, \(\HOD\) has the \(\kappa\)-cover property below \(\eta_0\) and correctly computes all successors of singular cardinals between \(\kappa\) and \(\eta_0\).

Our next proposition shows that our assumptions provide essentially no insight into the size of \(\eta_0\) relative to \(\kappa\). In particular, this shows that the strongly compact (or supercompact) HOD dichotomy theorem from \cite{Gol}
cannot be significantly improved.
\begin{theorem}\label{thm: omegastrongly}
    Suppose \(\delta\) is supercompact
    and the \(\HOD\) hypothesis fails. Then  for any cardinal \(\lambda\), there is a forcing extension \(V[G]\) in which \(\delta\) is indestructibly supercompact, the \(\HOD\) hypothesis fails, and $\mathcal{P}^{V[G]}(\lambda)\s \HOD^{V[G]}$. In particular, \(\eta_0^{V[G]} \geq \lambda\).
    \begin{proof}
        First force using the Laver preparation to produce a generic extension \(V[g]\) in which the supercompactness of \(\delta\) is indestructible under \(\delta\)-directed closed forcings. Then over \(V[g]\), use a \(\lambda^+\)-directed-closed forcing to produce an extension \(V[G] = V[g][h]\) such that every subset of $\lambda$ in $V[g]$ is coded into the continuum function of \(V[G]\) above \(\lambda\) (recall the coding poset from p.\pageref{McAloon}). Then \( \mathcal{P}^{V[G]}(\lambda) = \mathcal{P}^{V[g]}(\lambda) \subseteq \HOD^{V[G]}\), and in particular 
        \(\eta_0^{V[G]} \geq \alpha\).
    \end{proof}
\end{theorem}




\bibliographystyle{alpha} 
\bibliography{citations}
\end{document}